\documentclass[11pt]{article}
\usepackage{mathrsfs}
\usepackage{amsfonts}
\usepackage{}
\usepackage{amsmath,amssymb,amsthm,latexsym,amstext}
\usepackage[mathscr]{eucal}
\usepackage{rotate,graphics,epsfig,epstopdf}
\usepackage{float}
\usepackage{color}
\usepackage{subfigure}
\usepackage{indentfirst}
\usepackage{bm}

\newcommand{\be}{\begin{eqnarray}}
\newcommand{\ee}{\end{eqnarray}}
\newcommand{\by}{\begin{eqnarray*}}
\newcommand{\ey}{\end{eqnarray*}}
\newcommand{\bn}{\begin{enumerate}}
\newcommand{\en}{\end{enumerate}}
\newcommand{\ei}{\end{itemize}}

\newtheorem{theorem}{Theorem}[section]
\newtheorem{lemma}[theorem]{Lemma}

\newtheorem{remark}[theorem]{Remark}

\renewcommand{\theequation}{\arabic{section}.\arabic{equation}}

\numberwithin{equation}{section}

\begin{document}
\date{}
\title{\bf The Smoluchowski-Kramers approximation for a McKean-Vlasov equation subject to environmental noise with state-dependent friction  \footnote{This work was supported by
the Natural Science Foundation of Jiangsu Province, BK20230899 and
the National Natural Science Foundation of China, 11771207.
}}
\author{ Chungang Shi$^{1}$\footnote{Corresponding author, shichungang@njust.edu.cn } 
\hskip1cm Yan Lv$^{2}$\footnote{lvyan@njust.edu.cn}
\hskip1cm Wei Wang$^{3}$\footnote{wangweinju@nju.edu.cn} \\
\texttt{{\scriptsize $^{1,2}$School of Mathematics and Statistics, Nanjing University of Science and Technology,
Nanjing, 210094, P. R. China}}\\
\texttt{{\scriptsize $^{3}$Department of Mathematics, Nanjing University,
Nanjing, 210023, P. R. China}}}\maketitle
\begin{abstract}
The small mass limit is derived for a McKean-Vlasov equation subject to environmental noise with state-dependent friction. By applying the averaging approach to a non-autonomous stochastic slow-fast system with the microscopic and macroscopic scales,  the convergence in distribution is obtained.
\end{abstract}



\section{Introduction}\label{sec:intro}
  \setcounter{equation}{0}
  \renewcommand{\theequation}
{1.\arabic{equation}}
In the presence of interaction and confinement potentials, the evolution of the classical Newton dynamics for indistinguishable $N$-point particle system with mass $\epsilon$ in $\mathbb{R}^{d}$ is given by
\begin{eqnarray}\label{Intro:1.1}
\epsilon\ddot{X}_{i}+\gamma\dot{X}_{i}=-\nabla V(X_{i})-\frac{1}{N}\sum_{j=1}^{N}\nabla K(X_{i}-X_{j}), \quad i=1,2,\cdots,N,
\end{eqnarray}
where $X_{i}, \dot{X}_{i}\in\mathbb{R}^{d}$ are the position and velocity of the $i$-th particle respectively. The $\gamma>0$ describes the strength of linear damping in velocity and functions $V$ and $K:\mathbb{R}^{d}\to\mathbb{R}$ are the confinement potential and interaction potential respectively. As the number of particles $N$ tends to infinity, the microscopic descriptions of (\ref{Intro:1.1}) is approximately given by a McKean-Vlasov equation with the interaction potential replaced by a single averaged interaction~\cite{G},\cite{S},\cite{M},\cite{D}. In the real world, many systems may be disturbed by noise which is independent of each other acting on different particles. For example, particles are subject to the same space-dependent independent noise, in which case a $N$-particle system (\ref{Intro:1.1}) is described by
\begin{eqnarray}\label{equ:main}
&&\epsilon\ddot{X}_{t}^{i,\epsilon,N}+\gamma(X_{t}^{i,\epsilon,N})\dot{X}_{t}^{i,\epsilon,N}=\frac{1}{N}\sum_{j=1}^{\infty}K(X_{t}^{i,\epsilon,N}-X_{t}^{j,\epsilon,N})+\sum_{k=1}^{\infty}\sigma_{k}(X_{t}^{i,\epsilon,N})\circ\dot{B}_{t}^{k},\nonumber\\
&&X_{0}^{i,\epsilon,N}=X_{0}^{i},\quad V_{0}^{i,\epsilon,N}=V_{0}^{i}, \quad i=1,2,\cdots,N.
\end{eqnarray}
Here $\circ$ denotes the Stratonovich integral. The function $\gamma: \mathbb{R}^{d}\to \mathbb{R}$ is the state-dependent friction, $K, \sigma_{k}: \mathbb{R}^{d}\to \mathbb{R}^{d}$ are Lipschitz continuous functions defined later and $\{B^{k}\}_{k\in\mathbb{N}}$ are independent real-valued Brownian motions on a complete filtered probability space $(\Omega, \mathcal{F}, \mathcal{F}_{t}, \mathbb{P})$. Moreover, let $(X_{0}^{i},V_{0}^{i})$ be a sequence of i.i.d. random vectors in $\mathbb{R}^{d}$ which is independent of $\{B^{k}\}_{k\in\mathbb{N}}$ with law $\mu_{0}$.

Define the empirical measure 
\begin{equation*}
S_{t}^{N,\epsilon}\triangleq\frac{1}{N}\sum_{i=1}^{N}\delta_{(X_{t}^{i,\epsilon,N},V_{t}^{i,\epsilon,N})}
\end{equation*}
with $S_{0}^{N}=\frac{1}{N}\sum_{i=1}^{N}\delta_{(X_{0}^{i},V_{0}^{i})}$. Then the drift coefficient is rewritten as 
\begin{equation*}
\frac{1}{N}\sum_{j=1}^{\infty}K(X_{t}^{i,\epsilon,N}-X_{t}^{j,\epsilon,N})=K*S_{t}^{N,\epsilon}(X_{t}^{i,\epsilon,N}).
\end{equation*}
Similar approach as that for one order system~\cite{CF}, for fixed $\epsilon$, as $N\to\infty$, the limit of $S_{t}^{N,\epsilon}$ is shown to be $\mu_{t}^{\epsilon}$ which satisfies the following so-called mean field SPDE
\begin{eqnarray}\label{MFP}
d\mu_{t}^{\epsilon}+v\cdot\nabla_{x}\mu_{t}^{\epsilon}&=&\frac{\gamma(x)}{\epsilon}\nabla_{v}\cdot(v\mu_{t}^{\epsilon})dt-\frac{1}{\epsilon}\nabla_{v}\cdot((K*\mu_{t}^{\epsilon}(x))\mu_{t}^{\epsilon})dt\nonumber\\
&&-\frac{1}{\epsilon}\sum_{k=1}^{\infty}\nabla_{v}\cdot(\sigma_{k}(x)\mu_{t}^{\epsilon})dB_{t}^{k}+\frac{1}{\epsilon^{2}}\Delta_{v}\mu_{t}^{\epsilon}dt,
\end{eqnarray}
where $K*\mu_{t}^{\epsilon}(x)=\int_{\mathbb{R}^{2d}}K(x-y)\mu_{t}^{\epsilon}(y,dv)$. The measure-valued solutions of (\ref{MFP}) is a random probability measure~\cite{CF} whose randomness is due to environmental noise. In other words, the conditional law of $X_{t}^{i,\epsilon,N}$ given $\mathcal{F}_{t}^{B}$ that is the filtration associated to $\{B_{t}^{k}\}_{k\in\mathbb{N}}$ converges weakly to $\mu_{t}^{\epsilon}$.
Now, let ($X_{t}^{\epsilon,\mu^{\epsilon}},V_{t}^{\epsilon,\mu^{\epsilon}}$) be the solution of the following equations
\begin{eqnarray}\label{MEL}
dX_{t}^{\epsilon}&=&V_{t}^{\epsilon}dt,\nonumber\\
\epsilon dV_{t}^{\epsilon}&=&-\gamma(X_{t}^{\epsilon})V_{t}^{\epsilon}dt+K*\mu^{\epsilon}_{t}(X_{t}^{\epsilon})dt+\sum_{k=1}^{\infty}\sigma_{k}(X_{t}^{\epsilon})\circ dB_{t}^{k},\\
X_{0}^{\epsilon}&=&X_{0},\quad V_{0}^{\epsilon}=V_{0}\nonumber.
\end{eqnarray}
Given $\mu_{0}$ which satisfies assumption $(\mathbf{H_{4}})$ defined in next section, the push forward of $\mu_{0}$ with respect to the solution of stochastic characteristic equation (\ref{MEL}) namely $\mu_{t}^{\epsilon}(\omega)=(X_{t}^{\epsilon,\mu^{\epsilon}}(\omega),V_{t}^{\epsilon,\mu^{\epsilon}}(\omega))\#\mu_{0}(\omega)$ exactly satisfies (\ref{MFP}) in weak sense~\cite[Definition 11]{CF}. In fact, for any $\psi(x,y)\in C_{b}^{2}(\mathbb{R}^{2d})$, by It${\rm\hat{o}}$'s formula and the homogeneity assumption $(\mathbf{H_{3}})$ defined in next section,
\begin{eqnarray*}
d\psi(X_{t}^{\epsilon,\mu^{\epsilon}},V_{t}^{\epsilon,\mu^{\epsilon}})&=&\nabla_{x}\psi\cdot V_{t}^{\epsilon,\mu^{\epsilon}}dt+\nabla_{v}\psi\cdot(-\frac{\gamma(X_{t}^{\epsilon,\mu^{\epsilon}})}{\epsilon}V_{t}^{\epsilon,\mu^{\epsilon}}+\frac{1}{\epsilon}K*\mu_{t}^{\epsilon}(X_{t}^{\epsilon,\mu^{\epsilon}}))dt\\
&&+\frac{1}{2\epsilon^{2}}\Delta_{v}\psi dt+\sum_{k=1}^{\infty}\frac{1}{\epsilon}\nabla_{v}\psi\cdot\sigma_{k}(X_{t}^{\epsilon,\mu^{\epsilon}})dB_{t}^{k},
\end{eqnarray*}
integrating with respect to $\mu_{0}$, we have
\begin{eqnarray*}
d\langle\psi,\mu_{t}^{\epsilon}\rangle&=&\langle(\nabla_{x}\psi\cdot v-\frac{\gamma(x)}{\epsilon}\nabla_{v}\psi\cdot v+\frac{1}{\epsilon}\nabla_{v}\psi\cdot (K*\mu_{t}^{\epsilon})+\frac{1}{2\epsilon^{2}}\Delta_{v}\psi),\mu_{t}^{\epsilon}\rangle dt\\
&&+\frac{1}{\epsilon}\sum_{k=1}^{\infty}\langle\nabla_{v}\psi\cdot\sigma_{k},\mu_{t}^{\epsilon}\rangle dB_{t}^{k}.
\end{eqnarray*}
Then (\ref{MFP}) is derived by using stochastic Fubini theorem. 

Here we are concerned with the limit $\epsilon\to0$ on both sides of (\ref{MEL}). For a single particle, if the particle has a very small mass, then the small mass limit is desirable. There are numerous works on the small mass limit which is also known as the Smoluchowski-Kramers approximation~\cite{M1, SMD, HMVW}. These works all consider the small mass limit on the microscopic scale. In recent work~\cite{WLW}, Wang et al. derived a small mass limit for a stochastic $N$-particle system on the macroscopic scale by applying the averaging approach. However, they only considered the case of constant friction coefficients. In the paper, we consider a small mass limit with a state-dependent friction coefficient on the macroscopic scale by introducing some macroscopic quantities, which may be regarded as a generalization of ~\cite{HMVW} to $N$-particle system. We mainly study the macroscopic scale limit of the Vlasov-Fokker-Planck equation corresponding to (\ref{MFP}) so as to obtain the small mass limit of the equation (\ref{MEL}) in the sense of distribution. 

There are numerous works on the stochastic $N$-particle system~\cite{CC},\cite{WLW},\cite{BC}. Carrillo and Choi~\cite{CC} considered a general $N$ particle system without stochastic perturbation forces and made use of a discrete version of a modulated kinetic energy together with the bounded Lipschitz distance for measures to derive the mean field limit. Coghi and Flandoli~\cite{CF} proved the propagation of chaos for the first order stochastic $N$-particle system with environmental noise and the mean field limit PDE is shown to be a SPDE which is different from Wang et al.~\cite{WLW}. Liu and Wang~\cite{LW} considered a propagation of chaos for interacting particles system with common noise and a small mass limit is derived with constant friction. For more works, we refer to~\cite{Ro, HJNXZ, HQ}.

Furtherore, there are also some works on the macroscopic limit of Vlasov type equation. For example, Jabin~\cite{Ja} investigated the limit of some kinetic equation as the mass of the particle tends to zero. Karper~\cite{Ka} studied the hydrodynamic limit of a kinetic Cucker-Smale flocking model by means of a relative entropy method and the resulting asymptotic dynamics is an Euler-type flocking system. Huang~\cite{Hu} studied a kinetic Vlasov-Fokker-Planck equation and established a quantified estimate of the overdamped limit between the original equation and the aggregation-diffusion equation by adopting a probabilistic approach. For more results, we refer to~\cite{FS, FST, CT}. 

The rest of the paper is organized as follows. Some essential preliminaries and main results are given in section \ref{Pre}. The proof of the small mass limit is presented in the last section.

\section{Preliminary and main result}\label{Pre}
  \setcounter{equation}{0}
  \renewcommand{\theequation}
{2.\arabic{equation}}

Let $\mathbb{E}$ denote expectation with respect to $\mathbb{P}$ and $\mathbb{E}_{X}$ denotes the expectation with respect to the distribution of $X$.
Assuming that $\|\cdot\|$ denotes a vector or matrix norm and $\langle\cdot,\cdot\rangle$ represents the inner product on $\mathbb{R}^{d}$. 

Let $\mathcal{P}_{2}$ denote the sets consisting of Borel probability measure on $\mathbb{R}^{d}$ with 
\begin{equation*}
\int_{\mathbb{R}^{d}}\|x\|^{2}\mu(dx)<\infty,
\end{equation*}
for each $\mu\in\mathcal{P}_{2}$. For any $\mu$ and $\nu$ in $\mathcal{P}_{2}$, define the following Monge-Kantorovich (or Wasserstein) distance:
\begin{equation}\label{MKD}
dist_{MK,2}(\mu,\nu)=\inf_{\pi\in\Pi(\mu,\nu)}\Big[\int\int_{\mathbb{R}^{d}\times\mathbb{R}^{d}}\|x-y\|^{2}\pi(dx,dy)\Big]^{\frac{1}{2}},
\end{equation}
where $\Pi(\mu,\nu)$ denotes the set of Borel probability measures $\Pi$ on $\mathbb{R}^{d}\times\mathbb{R}^{d}$ with the first and second marginals $\mu$ and $\nu$. Equivalently, for $\mu, \nu\in \mathcal{P}_{2}$,
\begin{equation}\label{equ:2.101}
dist_{MK,2}(\mu,\nu)=\inf_{(X,\bar{X})}[\mathbb{E}\|X-\bar{X}\|^{2}]^{\frac{1}{2}}
\end{equation}
with random variables $X$ and $\bar{X}$ in $\mathbb{R}^{d}$ having laws $\mu$ and $\nu$ respectively.

Next we give some assumptions to system (\ref{MEL}).

$(\mathbf{H_{1}})$.$K: \mathbb{R}^{d}\to\mathbb{R}^{d}$ is Lipschitz continuous with Lipschitz constant $L_{K}$, that is 
$\|K(x)-K(y)\|\leq L_{K}\|x-y\|$, for any $x,y\in \mathbb{R}^{d}$.

$(\mathbf{H_{2}})$.$\gamma(\cdot): \mathbb{R}^{d}\to\mathbb{R}$ is a continuous and differential function whose derivative is bounded by $L_{\gamma}$ and there exists $\gamma_{0}$ and $\gamma_{1}$ such that 
\begin{eqnarray*}
0<\gamma_{0}\leq \gamma(x)\leq \gamma_{1},\quad x\in \mathbb{R}^{d}.
\end{eqnarray*}

$(\mathbf{H_{3}})$. {\rm (i)}$~\sigma_{k}: \mathbb{R}^{d}\to\mathbb{R}^{d}$ is measurable and satisfies $\sum_{k=1}^{\infty}\sup_{x\in\mathbb{R}^{d}}\|\sigma_{k}(x)\|\triangleq L_{\sigma}<\infty$. 

{\rm (ii)}~$\sigma_{k}$ is a $C^{2}$ divergence free vector fields, that is 
\begin{eqnarray*}
{\rm div}\sigma_{k}=0,\quad k\geq1.
\end{eqnarray*} 

Define the matrix-valued function $Q: \mathbb{R}^{d}\times\mathbb{R}^{d}\to \mathbb{R}^{d\times d}$ as 
\begin{eqnarray*}
Q^{ij}(x,y)\triangleq \sum_{k=1}^{\infty}\sigma_{k}^{i}(x)\sigma_{k}^{j}(y).
\end{eqnarray*}

{\rm (iii)} With a little abuse of notation, there exists a function $Q: \mathbb{R}^{d}\to\mathbb{R}^{d}\times\mathbb{R}^{d}$ such that 

{\rm (a)} $Q(x,y)=Q(x-y)$;

{\rm (b)} $Q(0)=I_{d}$;

{\rm (c)} $Q(\cdot)\in C^{2}(\mathbb{R}^{d})$ and $sup_{x\in\mathbb{R}^{d}}|\partial_{x_{i}x_{j}}^{2}Q(x)|<\infty$.

$(\mathbf{H_{4}})$. For the initial condition $\mu_{0}: \Omega\to\mathcal{P}_{2}(\mathbb{R}^{d})$ of equation (\ref{MFP}), let $\mu_{0}$ be $\mathcal{F}_{0}$-measurable.

\begin{remark}
Under the assumption $(\mathbf{H_{3}})$, the Stratonovich integral $\int_{0}^{t}\sigma_{k}(X_{s}^{i,\epsilon})\circ dB_{s}^{k}$ is equal to the It${\rm\hat{o}}$ integral $\int_{0}^{t}\sigma_{k}(X_{s}^{i,\epsilon})dB_{s}^{k}$, we refer to~\cite{CF}~or~\cite{LW} for details.
\end{remark}

Then we give the main result of the paper.
\begin{theorem}
Under the assumptions $(\mathbf{H_{1}})$-$(\mathbf{H_{4}})$, for any $T>s_{0}>0$, $\rho_{t}^{\epsilon}(x)\triangleq \int_{\mathbb{R}^{d}}\mu_{t}^{\epsilon}(x,v)dv$ converges weakly to $\rho_{t}$ for $s_{0}\leq t\leq T$ as $\epsilon\to0$ with
\begin{eqnarray}\label{main2}
\dot{\rho}_{t}(x)&=&-\nabla_{x}\cdot\Big[(\gamma(x)^{-1}(K*\rho_{t}(x))\rho_{t}(x)+\frac{1}{2}\sum_{k=1}^{\infty}\|\sigma_{k}(x)\|^{2}\nabla_{x}\gamma(x)\gamma(x)^{-3}\rho_{t}(x)\nonumber\\
&&+\sum_{k=1}^{\infty}\sigma_{k}(x)\rho_{t}(x)\gamma(x)^{-1}\dot{B}^{k}_{t}\Big]-\nabla_{x}^{2}\cdot\Big(\frac{1}{2}\sum_{k=1}^{\infty}\|\sigma_{k}(x)\|^{2}\gamma(x)^{-2}\rho_{t}(x)I_{d\times d}\Big),
\end{eqnarray}
which corresponds to the following SDE
\begin{eqnarray}\label{main-equ:2.2}
dX_{t}&=&\Big[\gamma(X_{t})^{-1}\tilde{E}K(X_{t}-\tilde{X}_{t})
+\frac{1}{2}\sum_{k=1}^{\infty}\|\sigma_{k}(X_{t})\|^{2}\nabla_{x}\gamma(X_{t})\gamma(X_{t})^{-3}\Big]dt\nonumber\\
&&+\sum_{k=1}^{\infty}\sigma_{k}(X_{t})\gamma(X_{t})^{-1}dB_{t}^{k}+\sqrt{\frac{1}{2}\sum_{k=1}^{\infty}\|\sigma_{k}(X_{t})\|^{2}\gamma(X_{t})^{-2}I_{d\times d}}dW_{t},
\end{eqnarray}
where $\tilde{X}_{t}$ is a version of $X_{t}$, $\tilde{E}$ is the expectation with respect to the distribution of $\tilde{X}_{t}$ and $B_{t}^{k}$ is a real-valued Brownian motion independent of standard $d$-dimensional Brownian motion $W_{t}$. Furthermore, $X_{t}^{\epsilon,\mu^{\epsilon}}$ converges  in distribution to $X_{t}$ as $\epsilon\to0$ in $L^{1}(\Omega, C(0,T;\mathbb{R}^{d}))$. 
\end{theorem}
\begin{remark}
The solution of equation (\ref{main2}) is actually a conditional probability distribution of $X_{t}$ satisfying equation (\ref{main-equ:2.2}) given the environmental noise $\{B_{t}^{k}\}_{k\in \mathbb{N}}$. Here by the conditional distribution $\rho_{t}$ of $X_{t}$ given $\mathcal{F}_{t}^{B}$, we mean that 
\begin{eqnarray*}
\mathbb{E}[X_{t}\in dx|\mathcal{F}_{t}^{B}]=\rho_{t}(dx).
\end{eqnarray*}
\end{remark}

\section{An averaging result: $\epsilon\to0$}\label{Aver}
  \setcounter{equation}{0}
  \renewcommand{\theequation}
{3.\arabic{equation}}
We need some uniform estimates on $X_{t}^{\epsilon,\mu^{\epsilon}}$ which is the solution of the equation (\ref{MEL}).
\begin{lemma}\label{SUB}
Assume that $(\mathbf{H_{1}})$-$(\mathbf{H_{3}})$ hold, then for every $T>0$, there exists a constant $C_{T}>0$ such that 
\begin{eqnarray*}
\mathbb{E}\sup_{0\leq t\leq T}\|X_{t}^{\epsilon,\mu^{\epsilon}}\|^{2}\leq C_{T}(1+\mathbb{E}\|X_{0}\|^{2}+\mathbb{E}\|V_{0}\|^{2}).
\end{eqnarray*}
\end{lemma}
\begin{proof}
Rewrite the equation (\ref{MEL}) as 
\begin{eqnarray}
&&dX_{t}^{\epsilon,\mu^{\epsilon}}=V_{t}^{\epsilon,\mu^{\epsilon}}dt\label{equ:main11}\\
&&dV_{t}^{\epsilon,\mu^{\epsilon}}=-\frac{1}{\epsilon}\gamma(X_{t}^{\epsilon,\mu^{\epsilon}})V_{t}^{\epsilon,\mu^{\epsilon}}dt+\frac{1}{\epsilon}(K\ast\mu_{t}^{\epsilon})(X_{t}^{\epsilon,\mu^{\epsilon}})dt+\frac{1}{\epsilon}\sum_{k=1}^{\infty}\sigma_{k}(X_{t}^{\epsilon,\mu^{\epsilon}})dB_{t}^{k}.\nonumber\label{equ:main12}
\end{eqnarray}
Introduce a process $Y_{t}^{\epsilon,\mu^{\epsilon}}$,
\begin{eqnarray}\label{equ:Y1}
\frac{dY_{t}^{\epsilon,\mu^{\epsilon}}}{dt}=-\frac{1}{\epsilon}\gamma(X_{t}^{\epsilon,\mu^{\epsilon}})Y_{t}^{\epsilon,\mu^{\epsilon}}, \quad Y_{0}^{\epsilon,\mu^{\epsilon}}=1.
\end{eqnarray}
Then
\begin{eqnarray*}
d[(Y_{t}^{\epsilon,\mu^{\epsilon}})^{-1}V_{t}^{\epsilon,\mu^{\epsilon}}]
&=&\frac{1}{\epsilon}(Y_{t}^{\epsilon,\mu^{\epsilon}})^{-1}(K\ast\mu_{t}^{\epsilon})(X_{t}^{\epsilon,\mu^{\epsilon}})dt+\frac{1}{\epsilon}\sum_{k=1}^{\infty}\sigma_{k}(X_{t}^{\epsilon,\mu^{\epsilon}})dB_{t}^{k}.
\end{eqnarray*}
Integrating from $0$ to $t$ yields
\begin{eqnarray*}
(Y_{t}^{\epsilon,\mu^{\epsilon}})^{-1}V_{t}^{\epsilon,\mu^{\epsilon}}=V_{0}+\frac{1}{\epsilon}\int_{0}^{t}(Y_{s}^{\epsilon,\mu^{\epsilon}})^{-1}(K\ast\mu_{s}^{\epsilon})(X_{s}^{\epsilon,\mu^{\epsilon}})ds+\frac{1}{\epsilon}\sum_{k=1}^{\infty}\int_{0}^{t}(Y_{s}^{\epsilon,\mu^{\epsilon}})^{-1}\sigma_{k}(X_{s}^{\epsilon,\mu^{\epsilon}})dB_{s}^{k},
\end{eqnarray*}
and
\begin{eqnarray}\label{VF}
V_{t}^{\epsilon,\mu^{\epsilon}}&=&Y_{t}^{\epsilon,\mu^{\epsilon}}V_{0}+\frac{1}{\epsilon}\int_{0}^{t}Y_{t}^{\epsilon,\mu^{\epsilon}}(Y_{s}^{\epsilon,\mu^{\epsilon}})^{-1}(K\ast\mu_{s}^{\epsilon})(X_{s}^{\epsilon,\mu^{\epsilon}})ds\nonumber\\
&&+\frac{1}{\epsilon}\sum_{k=1}^{\infty}\int_{0}^{t}Y_{t}^{\epsilon,\mu^{\epsilon}}(Y_{s}^{\epsilon,\mu^{\epsilon}})^{-1}\sigma_{k}(X_{s}^{\epsilon,\mu^{\epsilon}})dB_{s}^{k},
\end{eqnarray}
together with (\ref{equ:main11}),
\begin{eqnarray}\label{equ:X1}
X_{t}^{\epsilon,\mu^{\epsilon}}&=&X_{0}+\int_{0}^{t}Y_{s}^{\epsilon,\mu^{\epsilon}}V_{0}ds+\frac{1}{\epsilon}\int_{0}^{t}\int_{0}^{s}Y_{s}^{\epsilon,\mu^{\epsilon}}(Y_{u}^{\epsilon,\mu^{\epsilon}})^{-1}(K\ast\mu_{u}^{\epsilon})(X_{u}^{\epsilon,\mu^{\epsilon}})duds\nonumber\\
&&+\frac{1}{\epsilon}\sum_{k=1}^{\infty}\int_{0}^{t}\int_{0}^{s}Y_{s}^{\epsilon,\mu^{\epsilon}}(Y_{u}^{\epsilon,\mu^{\epsilon}})^{-1}\sigma_{k}(X_{u}^{\epsilon,\mu^{\epsilon}})dB_{u}^{k}ds\nonumber\\
&\triangleq& X_{0}+\sum_{j=1}^{3}I_{j}^{\epsilon}(t).
\end{eqnarray}
Next we estimate these terms respectively.
Let
\begin{eqnarray*}
\psi^{\epsilon}(t)=Y_{t}^{\epsilon,\mu^{\epsilon}}(Y_{s}^{\epsilon,\mu^{\epsilon}})^{-1}(K\ast\mu_{s}^{\epsilon})(X_{s}^{\epsilon,\mu^{\epsilon}}),
\end{eqnarray*}
then
\begin{eqnarray*}
\psi^{\epsilon}(s)=(K\ast\mu_{s}^{\epsilon})(X_{s}^{\epsilon,\mu^{\epsilon}}),
\end{eqnarray*}
and 
\begin{eqnarray*}
\frac{d\psi^{\epsilon}(t)}{dt}&=&\frac{dY_{t}^{\epsilon,\mu^{\epsilon}}}{dt}(Y_{s}^{\epsilon,\mu^{\epsilon}})^{-1}(K\ast\mu_{s}^{\epsilon})(X_{s}^{\epsilon,\mu^{\epsilon}})\\
&=&-\frac{1}{\epsilon}\gamma(X_{t}^{\epsilon,\mu^{\epsilon}})\psi^{\epsilon}(t).
\end{eqnarray*}
Solving it, we have
\begin{eqnarray*}
\psi^{\epsilon}(t)&=&\psi^{\epsilon}(s)e^{-\frac{1}{\epsilon}\int_{s}^{t}\gamma(X_{u}^{\epsilon,\mu^{\epsilon}})du}\\
&=&(K\ast\mu_{s}^{\epsilon})(X_{s}^{\epsilon,\mu^{\epsilon}})e^{-\frac{1}{\epsilon}\int_{s}^{t}\gamma(X_{u}^{\epsilon,\mu^{\epsilon}})du},
\end{eqnarray*}
and by $(\mathbf{H_{2}})$,
\begin{eqnarray*}
\|\psi^{\epsilon}(t)\|&\leq& \|(K\ast\mu_{s}^{\epsilon})(X_{s}^{\epsilon,\mu^{\epsilon}})\|e^{-\frac{1}{\epsilon}\int_{s}^{t}\gamma(X_{u}^{\epsilon,\mu^{\epsilon}})du}\\
&\leq&\|(K\ast\mu_{s}^{\epsilon})(X_{s}^{\epsilon,\mu^{\epsilon}})\|e^{-\frac{1}{\epsilon}\gamma_{0}(t-s)}.
\end{eqnarray*}
Furthermore, by (\ref{equ:Y1}) and $(\mathbf{H_{2}})$,
\begin{eqnarray}\label{YB}
Y_{t}^{\epsilon,\mu^{\epsilon}}=e^{-\frac{1}{\epsilon}\int_{0}^{t}\gamma(X_{s}^{\epsilon,\mu^{\epsilon}})ds},
\end{eqnarray}
and
\begin{eqnarray}\label{YY1}
\|Y_{t}^{\epsilon,\mu^{\epsilon}}V_{0}\|\leq \|V_{0}\|e^{-\frac{1}{\epsilon}\gamma_{0}t},\nonumber\\
|Y_{t}^{\epsilon,\mu^{\epsilon}}(Y_{s}^{\epsilon,\mu^{\epsilon}})^{-1}|\leq e^{-\frac{1}{\epsilon}\gamma_{0}(t-s)}.
\end{eqnarray}
Then for $I_{1}^{\epsilon}(t)$, we have
\begin{eqnarray}\label{equ:XE1}
\|I_{1}^{\epsilon}(t)\|\leq \int_{0}^{t}\|Y_{s}^{\epsilon,\mu^{\epsilon}}V_{0}\|ds\leq \|V_{0}\|\int_{0}^{t}e^{-\frac{1}{\epsilon}\gamma_{0}s}ds\leq \frac{\epsilon}{\gamma_{0}}\|V_{0}\|.
\end{eqnarray}
For $I_{2}^{\epsilon}(t)$, exchanging the order of integrations and together with $(\mathbf{H_{2}})$,
\begin{eqnarray}\label{equ:XE2}
\|I_{2}^{\epsilon}(t)\|&\leq& \frac{1}{\epsilon}\int_{0}^{t}\int_{0}^{s}\|Y_{s}^{\epsilon,\mu^{\epsilon}}(Y_{u}^{\epsilon,\mu^{\epsilon}})^{-1}(K\ast\mu_{u}^{\epsilon})(X_{u}^{\epsilon,\mu^{\epsilon}})\|duds\nonumber\\
&\leq&\frac{1}{\epsilon}\int_{0}^{t}\|(K\ast\mu_{u}^{\epsilon})(X_{u}^{\epsilon,\mu^{\epsilon}})\|\int_{u}^{t}e^{-\frac{1}{\epsilon}\gamma_{0}(s-u)}dsdu\nonumber\\
&\leq&\frac{1}{\gamma_{0}}\int_{0}^{t}\|(K\ast\mu_{u}^{\epsilon})(X_{u}^{\epsilon,\mu^{\epsilon}})\|du.
\end{eqnarray}
Note that
\begin{eqnarray}\label{equ:K1}
\|(K\ast\mu_{s}^{\epsilon})(X_{s}^{\epsilon,\mu^{\epsilon}})\|&=&\Big\|\int_{\mathbb{R}^{2d}}K(X_{s}^{\epsilon,\mu^{\epsilon}}-x)d\mu_{s}^{\epsilon}(x,v)\Big\|\nonumber\\
&\leq&\int_{\mathbb{R}^{2d}}\|K(X_{s}^{\epsilon,\mu^{\epsilon}}-x)-K(0)\|d\mu_{s}^{\epsilon}(x,v)+\int_{\mathbb{R}^{2d}}\|K(0)\|d\mu_{s}^{\epsilon}(x,v)\nonumber\\
&\leq& C_{L_{K}}(1+\|X_{s}^{\epsilon,\mu^{\epsilon}}\|+\mathbb{E}[\|X_{s}^{\epsilon,\mu^{\epsilon}}\||\mathcal{F}_{s}^{B}]).
\end{eqnarray}
Then combining (\ref{equ:XE2}) with (\ref{equ:K1}),
\begin{eqnarray*}
\|I_{2}^{\epsilon}(t)\|&\leq&C_{L_{K},\gamma_{0}}\int_{0}^{t}(1+\|X_{s}^{\epsilon,\mu^{\epsilon}}\|+\mathbb{E}[\|X_{s}^{\epsilon,\mu^{\epsilon}}\||\mathcal{F}_{s}^{B}])ds.
\end{eqnarray*}
Moreover,
\begin{eqnarray}\label{equ:XE2-1}
\mathbb{E}\sup_{0\leq t\leq T}\|I_{2}^{\epsilon}(t)\|^{2}&\leq& C_{L_{K},\gamma_{0},T}+C_{L_{K},\gamma_{0},T}\int_{0}^{T}\mathbb{E}\sup_{0\leq s\leq t}\|X_{s}^{\epsilon,\mu^{\epsilon}}\|^{2}dt.
\end{eqnarray}
Note that, by (\ref{YY1}) and $(\mathbf{H_{2}})$,
\begin{eqnarray*}
\frac{1}{\epsilon}\int_{u}^{t}|Y_{s}^{\epsilon,\mu^{\epsilon}}(Y_{u}^{\epsilon,\mu^{\epsilon}})^{-1}|ds\leq \frac{1}{\epsilon}\int_{u}^{t}e^{-\frac{1}{\epsilon}\gamma_{0}(s-u)}ds
\leq\frac{1}{\gamma_{0}},
\end{eqnarray*}
then for $I_{3}^{\epsilon}(t)$, exchanging the order of integrations, utilizing the Burkh${\rm\ddot{o}}$lder-Davies-Gundy inequality as well as $(\mathbf{H_{3}})$,
\begin{eqnarray}\label{equ:XE3}
\mathbb{E}\sup_{0\leq t\leq T}\|I_{3}^{\epsilon}(t)\|^{2}
&=&\mathbb{E}\sup_{0\leq t\leq T}\Big\|\frac{1}{\epsilon}\sum_{k=1}^{\infty}\int_{0}^{t}\int_{u}^{t}e^{-\frac{1}{\epsilon}\int_{u}^{s}\gamma(X_{r}^{\epsilon,\mu^{\epsilon}})dr} ds\sigma_{k}(X_{u}^{\epsilon,\mu^{\epsilon}})dB_{u}^{k}\Big\|^{2}\nonumber\\
&\leq&4\mathbb{E}\int_{0}^{T}\sup_{u}\sum_{k=1}^{\infty}\|\sigma_{k}(u)\|^{2}\Big(\frac{1}{\epsilon}\int_{u}^{t}e^{-\frac{1}{\epsilon}\int_{u}^{s}\gamma(X_{r}^{\epsilon,\mu^{\epsilon}})dr} ds\Big)^{2}du\nonumber\\
&\leq&C_{T,\gamma_{0},L_{\sigma}}.
\end{eqnarray}
Combining (\ref{equ:X1}) with (\ref{equ:XE1}), (\ref{equ:XE2}) and (\ref{equ:XE3}),
\begin{eqnarray*}
\mathbb{E}\sup_{0\leq t\leq T}\|X_{t}^{\epsilon,\mu^{\epsilon}}\|^{2}&\leq&C_{T,\gamma_{0},L_{\sigma},L_{K}}(1+\mathbb{E}\|V_{0}\|^{2})+C_{T,\gamma_{0},L_{\sigma},L_{K}}\int_{0}^{T}\mathbb{E}\sup_{0\leq s\leq t}\|X_{s}^{\epsilon,\mu^{\epsilon}}\|^{2}dt,
\end{eqnarray*}
Gronwall's inequality yields
\begin{eqnarray*}
\mathbb{E}\sup_{0\leq t\leq T}\|X_{t}^{\epsilon,\mu^{\epsilon}}\|^{2}&\leq& C_{T,\gamma_{0},L_{\sigma},L_{K}}(1+\mathbb{E}\|V_{0}\|^{2}+\mathbb{E}\|X_{0}\|^{2}).
\end{eqnarray*}
\end{proof}

Next we show that $X_{t}^{\epsilon,\mu^{\epsilon}}$ is H${\rm\ddot{o}}$lder continuous. This together with Lemma \ref{SUB}, we then obtain the tightness of $X_{t}^{\epsilon,\mu^{\epsilon}}$ in $C(0,T;\mathbb{R}^{d})$~\cite{CSa}.
\begin{lemma}\label{SC}
Assume that $(\mathbf{H_{1}})$-$(\mathbf{H_{3}})$ hold, then for every $T>0$, there exists a constant $C_{T}$, such that for any $0\leq s\leq t\leq T$,
\begin{eqnarray*}
\mathbb{E}\|X_{t}^{\epsilon,\mu^{\epsilon}}-X_{s}^{\epsilon,\mu^{\epsilon}}\|^{2}\leq C_{T}|t-s|,
\end{eqnarray*}
where $C_{T}$ depends on other parameters $\gamma_{0}, L_{K}, L_{\sigma}$ and initial values.
\end{lemma}
\begin{proof}
For any $0\leq s\leq t\leq T$, by (\ref{VF}), we have
\begin{eqnarray}\label{RZ}
&&\mathbb{E}\|X_{t}^{\epsilon,\mu^{\epsilon}}-X_{s}^{\epsilon,\mu^{\epsilon}}\|^{2}=\mathbb{E}\Big\|\int_{s}^{t}V_{u}^{\epsilon,\mu^{\epsilon}}du\Big\|^{2}\nonumber\\
&&\leq 3\mathbb{E}\Big\|\int_{s}^{t}Y_{u}^{\epsilon,\mu^{\epsilon}}V_{0}du\Big\|^{2}+3\mathbb{E}\Big\|\frac{1}{\epsilon}\int_{s}^{t}\int_{0}^{u}Y_{u}^{\epsilon,\mu^{\epsilon}}(Y_{r}^{\epsilon,\mu^{\epsilon}})^{-1}(K\ast\mu_{r}^{\epsilon})(X_{r}^{\epsilon,\mu^{\epsilon}})drdu\Big\|^{2}\nonumber\\
&&\quad+3\mathbb{E}\Big\|\frac{1}{\epsilon}\sum_{k=1}^{\infty}\int_{s}^{t}\int_{0}^{u}Y_{u}^{\epsilon,\mu^{\epsilon}}(Y_{r}^{\epsilon,\mu^{\epsilon}})^{-1}\sigma_{k}(X_{r}^{\epsilon,\mu^{\epsilon}})dB_{r}^{k}du\Big\|^{2}\nonumber\\
&&\triangleq R_{1}^{\epsilon}(t)+R_{2}^{\epsilon}(t)+R_{3}^{\epsilon}(t).
\end{eqnarray}
Next we estimate the three terms respectively.

For $R_{1}^{\epsilon}(t)$, by $(\mathbf{H_{2}})$, (\ref{YB}) and the H${\rm\ddot{o}}$lder inequality, we have
\begin{eqnarray}\label{R1}
R_{1}^{\epsilon}(t)&\leq& 3\mathbb{E}\Big(\int_{s}^{t}e^{-\frac{1}{\epsilon}\int_{0}^{u}\gamma(X_{r}^{\epsilon,\mu^{\epsilon}})dr}\|V_{0}\|du\Big)^{2}\nonumber\\
&\leq& C_{\gamma_{0},T}|t-s|\cdot\mathbb{E}\|V_{0}\|^{2}.
\end{eqnarray}
For $R_{2}^{\epsilon}(t)$, by (\ref{YB}), $(\mathbf{H_{1}})$, $(\mathbf{H_{2}})$, (\ref{equ:K1}) and the H${\rm\ddot{o}}$lder inequality,
\begin{eqnarray}\label{R2}
R_{2}^{\epsilon}(t)
&\leq&\frac{C_{L_{K}}}{\epsilon^{2}}\mathbb{E}\int_{s}^{t}\Big(\int_{0}^{u}|Y_{u}^{\epsilon,\mu^{\epsilon}}(Y_{r}^{\epsilon,\mu^{\epsilon}})^{-1}|(1+\|X_{r}^{\epsilon,\mu^{\epsilon}}\|+\mathbb{E}[\|X_{r}^{\epsilon,\mu^{\epsilon}}\||\mathcal{F}_{r}^{B}])dr\Big)^{2}du\cdot |t-s|\nonumber\\
&\leq&\frac{C_{L_{K}}}{\epsilon^{2}}(1+\mathbb{E}\sup_{0\leq t\leq T}\|X_{t}^{\epsilon,\mu^{\epsilon}}\|^{2})\int_{s}^{t}\Big(\int_{0}^{u}e^{-\frac{1}{\epsilon}\gamma_{0}(u-r)}dr\Big)^{2}du\cdot |t-s|\nonumber\\
&\leq&C_{\gamma_{0},L_{K},T}|t-s|(1+\mathbb{E}\|X_{0}\|^{2}+\mathbb{E}\|V_{0}\|^{2}).
\end{eqnarray}
For $R_{3}^{\epsilon}(t)$, by the H${\rm\ddot{o}}$lder inequality, the mean value theorem of integrals, there exists $\theta\in[s,t]$ such that $\int_{s}^{t}Y_{u}^{\epsilon,\mu^{\epsilon}}(Y_{r}^{\epsilon,\mu^{\epsilon}})^{-1}du=Y_{\theta}^{\epsilon,\mu^{\epsilon}}(Y_{r}^{\epsilon,\mu^{\epsilon}})^{-1}(t-s)$ and exchanging the order of integrals, we obtain
\begin{eqnarray}\label{R3}
R_{3}^{\epsilon}(t)
&=&\frac{3}{\epsilon^{2}}\mathbb{E}\Big\|\sum_{k=1}^{\infty}\int_{0}^{t}\int_{r}^{t}Y_{u}^{\epsilon,\mu^{\epsilon}}(Y_{r}^{\epsilon,\mu^{\epsilon}})^{-1}\sigma_{k}(X_{r}^{\epsilon,\mu^{\epsilon}})dB_{r}^{k}du\nonumber\\
&&-\sum_{k=1}^{\infty}\int_{0}^{s}\int_{r}^{s}Y_{u}^{\epsilon,\mu^{\epsilon}}(Y_{r}^{\epsilon,\mu^{\epsilon}})^{-1}\sigma_{k}(X_{r}^{\epsilon,\mu^{\epsilon}})dB_{r}^{k}du\Big\|^{2}\nonumber\\
&\leq&\frac{6}{\epsilon^{2}}\mathbb{E}\Big\|\sum_{k=1}^{\infty}\int_{0}^{s}\Big(\int_{s}^{t}Y_{u}^{\epsilon,\mu^{\epsilon}}(Y_{r}^{\epsilon,\mu^{\epsilon}})^{-1}du\Big)\sigma_{k}(X_{r}^{\epsilon,\mu^{\epsilon}})dB_{r}^{k}\Big\|^{2}\nonumber\\
&&+\frac{6}{\epsilon^{2}}\mathbb{E}\Big\|\sum_{k=1}^{\infty}\int_{s}^{t}\int_{r}^{t}Y_{u}^{\epsilon,\mu^{\epsilon}}(Y_{r}^{\epsilon,\mu^{\epsilon}})^{-1}du\sigma_{k}(X_{r}^{\epsilon,\mu^{\epsilon}})dB_{r}^{k}\Big\|\nonumber\\
&\leq&\frac{6}{\epsilon^{2}}|t-s|^{2}\mathbb{E}\Big\|\sum_{k=1}^{\infty}\int_{0}^{s}Y_{\theta}^{\epsilon,\mu^{\epsilon}}(Y_{r}^{\epsilon,\mu^{\epsilon}})^{-1}\sigma_{k}(X_{r}^{\epsilon,\mu^{\epsilon}})dB_{r}^{k}\Big\|^{2}\nonumber\\
&&+\frac{24}{\epsilon^{2}}\mathbb{E}\sum_{k=1}^{\infty}\int_{s}^{t}\Big\|\int_{r}^{t}Y_{u}^{\epsilon,\mu^{\epsilon}}(Y_{r}^{\epsilon,\mu^{\epsilon}})^{-1}\sigma_{k}(X_{r}^{\epsilon,\mu^{\epsilon}})du\Big\|^{2}dr\nonumber\\
&\leq&\frac{24}{\epsilon^{2}}|t-s|^{2}\mathbb{E}\sum_{k=1}^{\infty}\int_{0}^{s}e^{-\frac{2}{\epsilon}\gamma_{0}(\theta-r)}\|\sigma_{k}(X_{r}^{\epsilon,\mu^{\epsilon}})\|^{2}dr\nonumber\\
&&+\frac{24}{\epsilon^{2}}\mathbb{E}\sum_{k=1}^{\infty}\int_{s}^{t}\Big(\int_{r}^{t}e^{-\frac{1}{\epsilon}\gamma_{0}(u-r)}\|\sigma_{k}(X_{r}^{\epsilon,\mu^{\epsilon}})\|du\Big)^{2}dr\nonumber\\
&\leq&C_{\gamma_{0},T}\sum_{k=1}^{\infty}\sup_{x}\|\sigma_{k}(x)\|^{2}\frac{1}{\epsilon}e^{-\frac{2}{\epsilon}\gamma_{0}(\theta-s)}\cdot|t-s|^{2}+C_{\gamma_{0},T}\sup_{x}\sum_{k=1}^{\infty}\|\sigma_{k}(x)\|^{2}|t-s|\nonumber\\
&\leq&C_{\gamma_{0},T,L_{\sigma}}|t-s|.
\end{eqnarray}
Then by (\ref{RZ}), (\ref{R1}), (\ref{R2}) and  (\ref{R3}), we obtain
\begin{eqnarray*}
\mathbb{E}\|X_{t}^{\epsilon,\mu^{\epsilon}}-X_{s}^{\epsilon,\mu^{\epsilon}}\|^{2}\leq C_{\gamma_{0},L_{K},L_{\sigma},T}|t-s|(1+\mathbb{E}\|X_{0}\|^{2}+\mathbb{E}\|V_{0}\|^{2}).
\end{eqnarray*}
\end{proof}

We also need estimate for $V_{t}^{\epsilon,\mu^{\epsilon}}$ with fixed $X_{t}^{\epsilon,\mu^{\epsilon}}=x, \mu_{t}^{\epsilon}=\mu$. For this, let $V_{t}^{\epsilon,\mu}$ satisfy the following SDEs,
\begin{eqnarray*}
\epsilon\dot{V}_{t}^{\epsilon,\mu,x}=-\gamma(x)V_{t}^{\epsilon,\mu,x}+ K*\mu(x)+\sum_{k=1}^{\infty}\sigma_{k}(x)\dot{B}_{t}^{k},
\end{eqnarray*}
then for $t\geq 0$, we have $\mathbb{E}^{x}\|V_{t}^{\epsilon,\mu}\|^{2}=\mathbb{E}\|V_{t}^{\epsilon,\mu,x}\|^{2}$ and the following result holds.
\begin{lemma}\label{LVV}
Suppose that $(\mathbf{H_{1}})$-$(\mathbf{H_{4}})$ hold, then for any $x\in \mathbb{R}^{d}$, and fixed $X_{t}^{\epsilon,\mu^{\epsilon}}=x, \mu_{t}^{\epsilon}=\mu$, 
\begin{eqnarray*}
\mathbb{E}^{x}(\epsilon V_{t}^{\epsilon,\mu}\otimes V_{t}^{\epsilon,\mu})=\frac{1}{2}\sum_{k=1}^{\infty}\|\sigma_{k}(x)\|^{2}\gamma(x)^{-1}I_{d\times d}+\epsilon C(x,t),
\end{eqnarray*}
where $\|C(x,t)\|\leq M(1+\|x\|^{2})$.
\end{lemma}
\begin{proof}
By the Duhamel principle, 
\begin{eqnarray*}
V_{t}^{\epsilon,\mu}=e^{-\frac{\gamma(x)}{\epsilon}t}V_{0}+\frac{1}{\epsilon}\int_{0}
^{t}e^{-\frac{\gamma(x)}{\epsilon}(t-s)}(K*\mu(x))ds+\frac{1}{\epsilon}\sum_{k=1}^{\infty}\int_{0}
^{t}e^{-\frac{\gamma(x)}{\epsilon}(t-s)}\sigma_{k}(x)dB_{s}^{k}.
\end{eqnarray*}
Taking expectation on both sides,
\begin{eqnarray}\label{Ex}
\mathbb{E}^{x}V_{t}^{\epsilon,\mu}=e^{-\frac{\gamma(x)}{\epsilon}t}\mathbb{E}V_{0}+\gamma(x)^{-1}(1-e^{-\frac{\gamma(x)}{\epsilon}t})(K*\mu(x)).
\end{eqnarray}
Applying It${\rm\hat{o}}$'s formula to $\frac{1}{2}(V_{t}^{\epsilon,\mu}\otimes V_{t}^{\epsilon,\mu})$,
\begin{eqnarray*}
\frac{d}{dt}\mathbb{E}^{x}(\epsilon V_{t}^{\epsilon,\mu}\otimes V_{t}^{\epsilon,\mu})
&=&-2\frac{\gamma(x)}{\epsilon}\mathbb{E}^{x}(\epsilon V_{t}^{\epsilon,\mu}\otimes V_{t}^{\epsilon,\mu})+\mathbb{E}^{x}V_{t}^{\epsilon,\mu}\otimes (K*\mu(x))\\
&&+(K*\mu(x))\otimes\mathbb{E}^{x}V_{t}^{\epsilon,\mu}+\frac{1}{\epsilon}\sum_{k=1}^{\infty}\|\sigma_{k}(x)\|^{2}I_{d\times d}.\\
\end{eqnarray*}
By the Duhamel principle and (\ref{Ex}), we have
\begin{eqnarray}\label{VVZ}
\mathbb{E}^{x}(\epsilon V_{t}^{\epsilon,\mu}\otimes V_{t}^{\epsilon,\mu})
&=&e^{-2\frac{\gamma(x)}{\epsilon}t}\mathbb{E}^{x}(\epsilon V_{0}\otimes V_{0})+\int_{0}^{t}e^{-2\frac{\gamma(x)}{\epsilon}(t-s)}\mathbb{E}^{x}V_{s}^{\epsilon,\mu}\otimes (K*\mu(x))ds\nonumber\\
&&+\int_{0}^{t}e^{-2\frac{\gamma(x)}{\epsilon}(t-s)}(K*\mu(x))\otimes\mathbb{E}^{x}V_{s}^{\epsilon,\mu}ds\nonumber\\
&&+\frac{1}{\epsilon}\sum_{k=1}^{\infty}\|\sigma_{k}(x)\|^{2}\int_{0}^{t}e^{-\frac{\gamma(x)}{\epsilon}(t-s)}ds\cdot I_{d\times d}\nonumber\\
&=&e^{-2\frac{\gamma(x)}{\epsilon}t}\mathbb{E}^{x}(\epsilon V_{0}\otimes V_{0})+\int_{0}^{t}e^{-2\frac{\gamma(x)}{\epsilon}(t-s)}e^{-\frac{\gamma(x)}{\epsilon}s}\mathbb{E}V_{0}\otimes (K*\mu(x))ds\nonumber\\
&&+\int_{0}^{t}e^{-2\frac{\gamma(x)}{\epsilon}(t-s)}e^{-\frac{\gamma(x)}{\epsilon}s}(K*\mu(x))\otimes\mathbb{E}V_{0}ds\nonumber\\
&&+2\int_{0}^{t}e^{-2\frac{\gamma(x)}{\epsilon}(t-s)}\gamma(x)^{-1}(1-e^{-\frac{\gamma(x)}{\epsilon}t})(K*\mu(x))\otimes (K*\mu(x))ds\nonumber\\
&&+\frac{1}{2}\sum_{k=1}^{\infty}\|\sigma_{k}(x)\|^{2}(1-e^{-\frac{\gamma(x)}{\epsilon}t})\cdot I_{d\times d}\nonumber\\
&\triangleq&\sum_{j=1}^{4}J_{i}^{\epsilon}(t)+\frac{1}{2}\sum_{k=1}^{\infty}\|\sigma_{k}(x)\|^{2}\gamma(x)^{-1}\cdot I_{d\times d}.
\end{eqnarray}
Next we estimate these terms respectively. In fact, for $J_{1}^{\epsilon}(t)$, by $(\mathbf{H_{2}})$,
\begin{eqnarray}\label{VV1}
\|J_{1}^{\epsilon}(t)\|=\|e^{-2\frac{\gamma(x)}{\epsilon}t}\mathbb{E}^{x}(\epsilon V_{0}\otimes V_{0})\|\leq \epsilon e^{-\frac{2\gamma_{0}}{\epsilon}t}\mathbb{E}\|V_{0}\|^{2}.
\end{eqnarray}
Similarly,
\begin{eqnarray}\label{VV2}
\|J_{2}^{\epsilon}(t)\|&=&\Big\|\int_{0}^{t}e^{-2\frac{\gamma(x)}{\epsilon}(t-s)}e^{-\frac{\gamma(x)}{\epsilon}s}\mathbb{E}V_{0}\otimes (K*\mu(x))ds\nonumber\\
&&+\int_{0}^{t}e^{-2\frac{\gamma(x)}{\epsilon}(t-s)}e^{-\frac{\gamma(x)}{\epsilon}s}(K*\mu(x))\otimes\mathbb{E}V_{0}ds\Big\|\nonumber\\
&\leq&2\int_{0}^{t}e^{-2\frac{\gamma_{0}}{\epsilon}(t-s)}e^{-\frac{\gamma_{0}}{\epsilon}s}\mathbb{E}\|V_{0}\|\|K*\mu(x)\|ds\nonumber\\
&\leq&\epsilon C_{\gamma_{0}}\mathbb{E}\|V_{0}\|\|K*\mu(x)\|,
\end{eqnarray}
and
\begin{eqnarray}\label{VV3}
\|J_{3}^{\epsilon}(t)\|&=&\Big\|2\int_{0}^{t}e^{-2\frac{\gamma(x)}{\epsilon}(t-s)}\gamma(x)^{-1}(1-e^{-\frac{\gamma(x)}{\epsilon}t})(K*\mu(x))\otimes (K*\mu(x))ds\Big\|\nonumber\\
&\leq&C_{\gamma_{0}}\|K*\mu(x)\|^{2}\int_{0}^{t}e^{-2\frac{\gamma_{0}}{\epsilon}(t-s)}ds\nonumber\\
&\leq&\epsilon C_{\gamma_{0}}\|K*\mu(x)\|^{2}.
\end{eqnarray}
Furthermore,
\begin{eqnarray}\label{VV4}
\|J_{4}^{\epsilon}(t)\|&=&\Big\|\frac{1}{2}\sum_{k=1}^{\infty}\|\sigma_{k}(x)\|^{2}\gamma(x)^{-1}e^{-\frac{\gamma(x)}{\epsilon}t}\cdot I_{d\times d}\Big\|\nonumber\\
&\leq& \frac{\sqrt{d}}{2}\sum_{k=1}^{\infty}\|\sigma_{k}(x)\|^{2}|\gamma(x)^{-1}|e^{-\frac{\gamma_{0}}{\epsilon}t}\leq\epsilon C_{\gamma_{0},L_{\sigma}}.
\end{eqnarray}
Then by (\ref{VVZ}), (\ref{VV1}), (\ref{VV2}), ((\ref{VV3})) and (\ref{VV4}), we obtain
\begin{eqnarray*}\label{VVZ1}
\mathbb{E}^{x}(\epsilon V_{t}^{\epsilon,\mu}\otimes V_{t}^{\epsilon,\mu})&=&\frac{1}{2}\sum_{k=1}^{\infty}\|\sigma_{k}(x)\|^{2}\gamma(x)^{-1}\cdot I_{d\times d}+\epsilon C(x,t),
\end{eqnarray*}
where $\|C(x,t)\|\leq M(1+\|x\|^{2})$.
\end{proof}

By Lemma \ref{SUB} and Lemma \ref{SC}, we obtain that $\{X_{t}^{\epsilon,\mu^{\epsilon}}\}_{0< \epsilon\leq1}$ is tight in space $C(0,T;\mathbb{R}^{d})$. Next we pass the limit in (\ref{MEL}) as $\epsilon\to0$.

Note that the law $\mu_{t}^{\epsilon}$ of $(X_{t}^{\epsilon,\mu^{\epsilon}},V_{t}^{\epsilon,\mu^{\epsilon}})$
satisfies 
\begin{eqnarray}\label{FPE}
d\mu_{t}^{\epsilon}+v\cdot\nabla_{x}\mu_{t}^{\epsilon}dt&=&\frac{\gamma(x)}{\epsilon}\nabla_{v}\cdot(v\mu_{t}^{\epsilon})dt-\frac{1}{\epsilon}\nabla_{v}\cdot((K*\mu_{t}^{\epsilon}(x))\mu_{t}^{\epsilon})dt\nonumber\\
&&-\frac{1}{\epsilon}\sum_{k=1}^{\infty}\nabla_{v}\cdot(\sigma_{k}(x)\mu_{t}^{\epsilon})dB_{t}^{k}+\frac{1}{\epsilon^{2}}\Delta_{v}\mu_{t}^{\epsilon}dt,
\end{eqnarray}
in the weak sense, that is, for $\Psi\in C_{0}^{\infty}(\mathbb{R}^{d}\times\mathbb{R}^{d})$,
\begin{eqnarray}\label{equ:3.2}
&&\langle \mu_{t}^{\epsilon},\Psi\rangle-\langle \mu_{0}^{\epsilon},\Psi\rangle\nonumber\\
&=&\int_{0}^{t}\int_{\mathbb{R}^{d}\times\mathbb{R}^{d}}\Big[v\nabla_{x}\Psi-\frac{\gamma(x)}{\epsilon}v+\frac{K*\rho_{s}^{\epsilon}(x)}{\epsilon}+\frac{1}{\epsilon}\sum_{k=1}^{\infty}\sigma_{k}(x)dB_{s}^{k}\nabla_{v}\Psi\nonumber\\
&&\quad\quad\quad\quad\quad\quad\quad\quad\quad\quad\quad\quad\quad\quad\quad+\frac{1}{\epsilon^{2}}\Delta_{v}\Psi\Big]\mu_{s}^{\epsilon}(x,v)dxdv,
\end{eqnarray}
where $\rho_{t}^{\epsilon}(x)=\int_{\mathbb{R}^{2d}}\mu_{t}^{\epsilon}(x,v)dv$ is the conditional law of $X_{t}^{\epsilon,\mu^{\epsilon}}$ with respect to $\mathcal{F}_{t}^{B}$.
For small $\epsilon>0$, the above equation (\ref{FPE}) is a singular system which is difficult to pass the limit $\epsilon\to0$. Here we present an averaging approach to derive such limit for $\mu_{t}^{\epsilon}$ explicitly.

For this, we define the local mass $\rho_{t}^{\epsilon}$, the marginal distribution of $\mu_{t}^{\epsilon}$, and local momentum $Y_{t}^{\epsilon}$ as
\begin{eqnarray*}
\rho_{t}^{\epsilon}(x)=\int_{\mathbb{R}^{d}}\mu_{t}^{\epsilon}(x,v)dv, \quad Y_{t}^{\epsilon}=\int_{\mathbb{R}^{d}}v\mu_{t}^{\epsilon}(x,v)dv,
\end{eqnarray*}
respectively.  Then in the weak sense,
\begin{eqnarray}\label{LZD}
d\rho_{t}^{\epsilon}&=&-\nabla_{x} Y_{t}^{\epsilon}dt\\
dY_{t}^{\epsilon}&=&-\frac{\gamma(x)}{\epsilon}Y_{t}^{\epsilon}dt+\frac{K*\rho_{t}^{\epsilon}(x)}{\epsilon}\rho_{t}^{\epsilon}dt\nonumber\\
&&-\nabla_{x}\cdot[\rho_{t}^{\epsilon}\mathbb{E}^{x}(V_{t}^{\epsilon,\mu^{\epsilon}}\otimes V_{t}^{\epsilon,\mu^{\epsilon}})]dt-\frac{1}{\epsilon}\sum_{k=1}^{\infty}\sigma_{k}(x)dB_{t}^{k}\rho_{t}^{\epsilon}.
\end{eqnarray}
Here we have used the following calculations
\begin{equation*}
\int_{\mathbb{R}^{d}}v\otimes v\mu_{t}^{\epsilon}(x,v)dv=\rho_{t}^{\epsilon}(x)\int_{\mathbb{R}^{d}}v\otimes v\frac{\mu_{t}^{\epsilon}(x,v)}{\rho_{t}^{\epsilon}(x)}dv=\rho_{t}^{\epsilon}(x)\mathbb{E}^{x}(V_{t}^{\epsilon,\mu^{\epsilon}}\otimes V_{t}^{\epsilon,\mu^{\epsilon}}).
\end{equation*}
We couple equation (\ref{MEL}) to the above stochastic slow-fast system which yields the following closed system
\begin{eqnarray}\label{HE}
d\rho_{t}^{\epsilon}&=&-\nabla_{x}Y_{t}^{\epsilon}\label{Macro1}dt,\\
dY_{t}^{\epsilon}&=&-\frac{\gamma(x)}{\epsilon}Y_{t}^{\epsilon}dt+\frac{K*\rho_{t}^{\epsilon}(x)}{\epsilon}\rho_{t}^{\epsilon}dt-\nabla_{x}\cdot[\rho_{t}^{\epsilon}\mathbb{E}^{x}(V_{t}^{\epsilon,\mu^{\epsilon}}\otimes V_{t}^{\epsilon,\mu^{\epsilon}})]dt\nonumber\\
&&-\frac{1}{\epsilon}\sum_{k=1}^{\infty}\sigma_{k}(x)dB_{t}^{k}\rho_{t}^{\epsilon},\label{Macro2}\\
dX_{t}^{\epsilon,\mu^{\epsilon}}&=&V_{t}^{\epsilon,\mu^{\epsilon}}dt,\quad \rho_{t}^{\epsilon}=\mathcal{L}(X_{t}^{\epsilon,\mu^{\epsilon}}|\mathcal{F}_{t}^{B}),\\
\epsilon dV_{t}^{\epsilon,\mu^{\epsilon}}&=&-\gamma(X_{t}^{\epsilon,\mu^{\epsilon}})V_{t}^{\epsilon,\mu^{\epsilon}}dt+K*\rho_{t}^{\epsilon}(X_{t}^{\epsilon,\mu^{\epsilon}})dt+\sum_{k=1}^{\infty}\sigma_{k}(X_{t}^{\epsilon,\mu^{\epsilon}})dB_{t}^{k}.\nonumber\\
\end{eqnarray}
Here $\mathbb{E}^{x}$ is the expectation with fixed $X_{t}^{\epsilon,\mu^{\epsilon}}=x\in\mathbb{R}^{d}$. We consider the equations (\ref{Macro1})-(\ref{Macro2}) in the weak sense as we just concerned with the weak convergence of $\rho_{t}^{\epsilon}$.
Next we apply an averaging approach to pass the limit $\epsilon\to0$. We need an auxiliary process $\{\hat{\rho}_{t}^{\epsilon}, \hat{Y}_{t}^{\epsilon}\}_{0\leq t\leq T}$. For this we divide the time interval $[0,T]$ into intervals of size $\delta>0$ as $0=t_{0}<t_{1}<\cdots<t_{k}<t_{k+1}<\cdots<t_{[\frac{T}{\delta}]+1}=T$, for $t_{k+1}-t_{k}=\delta, k=0,\cdots,[\frac{T}{\delta}]$ and for $t\in[t_{k},t_{k+1}]$,
\begin{eqnarray}\label{FE1}
d\hat{\rho}_{t}^{\epsilon}&=&-\nabla_{x}\hat{Y}_{t}^{\epsilon}dt,\quad \hat{\rho}_{t_{k}}^{\epsilon}=\rho_{t_{k}}^{\epsilon},\\
d\hat{Y}_{t}^{\epsilon}&=&-\frac{\gamma(x)}{\epsilon}\hat{Y}_{t}^{\epsilon}dt+\frac{K*\rho_{t_{k}}^{\epsilon}(x)}{\epsilon}\rho_{t_{k}}^{\epsilon}dt\nonumber\\
&&-\nabla_{x}\cdot[\rho_{t_{k}}^{\epsilon}\mathbb{E}^{x}(V_{t_{k}}^{\epsilon,\mu^{\epsilon}}\otimes V_{t_{k}}^{\epsilon,\mu^{\epsilon}})]dt+\frac{1}{\epsilon}\sum_{k=1}^{\infty}\sigma_{k}(x)dB_{t}^{k}\rho_{t_{k}}^{\epsilon}.\\
\hat{Y}_{t_{k}}^{\epsilon}&=&Y_{t_{k}}^{\epsilon},\nonumber
\end{eqnarray}
Next lemma gives an estimate of the difference between $\hat{Y}_{t}^{\epsilon}$ and $Y_{t}^{\epsilon}$.
\begin{lemma}\label{LC1}
Assume that $(\mathbf{H_{1}})$-$(\mathbf{H_{4}})$ are valid, for every $T>0$ and $\psi\in C_{0}^{\infty}(\mathbb{R}^{d};\mathbb{R}^{d})$, there is a constant $C_{T}>0$ such that
\begin{eqnarray*}
\sup_{0\leq t\leq T}\mathbb{E}|\langle Y_{t}^{\epsilon}-\hat{Y}_{t}^{\epsilon},\psi\rangle|\leq C_{T}\Big(\frac{\sqrt{\delta}}{\epsilon}+\frac{\delta}{\epsilon}+\frac{\delta}{\epsilon^{2}}\Big)\|\psi\|_{Lip},
\end{eqnarray*}
where $C_{T}$ also depends on parameters $\gamma_{1}, L_{K}, L_{\sigma}, L_{\gamma}, M$ and initial values.
\end{lemma}
\begin{proof}
Let $Z_{t}^{\epsilon}=Y_{t}^{\epsilon}-\hat{Y}_{t}^{\epsilon}$, then for $t\in[t_{k},t_{k+1}]$, we have
\begin{eqnarray*}
dZ_{t}^{\epsilon}&=&-\frac{\gamma(x)}{\epsilon}Z_{t}^{\epsilon}dt
+\frac{K*\rho_{t}^{\epsilon}(x)}{\epsilon}\rho_{t}^{\epsilon}(x)dt-\frac{K*\rho_{t_{k}}^{\epsilon}(x)}{\epsilon}\rho_{t_{k}}^{\epsilon}dt\\
&&-\frac{1}{\epsilon}\sum_{k=1}^{\infty}\sigma_{k}(x)dB_{t}^{k}(\rho_{t}^{\epsilon}-\rho_{t_{k}}^{\epsilon})\\
&&-\nabla_{x}\cdot[\rho_{t}^{\epsilon}\mathbb{E}^{x}(V_{t}^{\epsilon,\mu^{\epsilon}}\otimes V_{t}^{\epsilon,\mu^{\epsilon}})]dt
+\nabla_{x}\cdot[\rho_{t_{k}}^{\epsilon}\mathbb{E}^{x}(V_{t_{k}}^{\epsilon,\mu^{\epsilon}}\otimes V_{t_{k}}^{\epsilon,\mu^{\epsilon}})]dt.
\end{eqnarray*}
By the Duhamel principle, we have
\begin{eqnarray*}
Z_{t}^{\epsilon}&=&\frac{1}{\epsilon}\int_{t_{k}}^{t}e^{-\frac{\gamma(x)}{\epsilon}(t-s)}[K*\rho_{s}^{\epsilon}(x)\rho_{s}^{\epsilon}(x)-K*\rho_{t_{k}}^{\epsilon}(x)\rho_{t_{k}}^{\epsilon}(x)]ds\\
&&-\frac{1}{\epsilon}\sum_{l=1}^{\infty}\int_{t_{k}}^{t}e^{-\frac{\gamma(x)}{\epsilon}(t-s)}\sigma_{l}(x)(\rho_{s}^{\epsilon}-\rho_{t_{k}}^{\epsilon})dB_{s}^{l}\\
&&-\int_{t_{k}}^{t}e^{-\frac{\gamma(x)}{\epsilon}(t-s)}\nabla_{x}\cdot[\rho_{s}^{\epsilon}\mathbb{E}^{x}(V_{s}^{\epsilon,\mu^{\epsilon}}\otimes V_{s}^{\epsilon,\mu^{\epsilon}})-\rho_{t_{k}}^{\epsilon}\mathbb{E}^{x}(V_{t_{k}}^{\epsilon,\mu^{\epsilon}}\otimes V_{t_{k}}^{\epsilon,\mu^{\epsilon}})]ds,
\end{eqnarray*}
and 
\begin{eqnarray}\label{ZZ}
\langle Z_{t}^{\epsilon},\psi\rangle&=&\frac{1}{\epsilon}\int_{t_{k}}^{t}\langle e^{-\frac{\gamma(x)}{\epsilon}(t-s)}[K*\rho_{s}^{\epsilon}(x)\rho_{s}^{\epsilon}(x)-K*\rho_{t_{k}}^{\epsilon}(x)\rho_{t_{k}}^{\epsilon}(x)],\psi\rangle ds\nonumber\\
&&-\frac{1}{\epsilon}\sum_{l=1}^{\infty}\int_{t_{k}}^{t}\langle e^{-\frac{\gamma(x)}{\epsilon}(t-s)}\sigma_{l}(x)(\rho_{s}^{\epsilon}-\rho_{t_{k}}^{\epsilon}),\psi\rangle dB_{s}^{l}\nonumber\\
&&-\int_{t_{k}}^{t}\langle e^{-\frac{\gamma(x)}{\epsilon}(t-s)}\nabla_{x}\cdot[\rho_{s}^{\epsilon}\mathbb{E}^{x}(V_{s}^{\epsilon,\mu^{\epsilon}}\otimes V_{s}^{\epsilon,\mu^{\epsilon}})-\rho_{t_{k}}^{\epsilon}\mathbb{E}^{x}(V_{t_{k}}^{\epsilon,\mu^{\epsilon}}\otimes V_{t_{k}}^{\epsilon,\mu^{\epsilon}})],\psi\rangle ds\nonumber\\
&\triangleq&\frac{1}{\epsilon}\Pi_{1}^{\epsilon}(t)+\Pi_{2}^{\epsilon}(t)+\Pi_{3}^{\epsilon}(t).
\end{eqnarray}
By the definition of expectation, $(\mathbf{H_{1}})$ and the basic inequality: $|e^{x}-e^{y}|\leq |x-y|(e^{x}+e^{y}), x,y\in \mathbb{R}$, we have
\begin{eqnarray*}
&&\mathbb{E}|\langle e^{-\frac{\gamma(x)}{\epsilon}(t-s)}[K*\rho_{s}^{\epsilon}(x)\rho_{s}^{\epsilon}(x)-K*\rho_{t_{k}}^{\epsilon}(x)\rho_{t_{k}}^{\epsilon}(x)],\psi\rangle|\\
&=&\mathbb{E}|\mathbb{E}_{X_{s}^{\epsilon,\mu^{\epsilon}}}[e^{-\frac{\gamma(X_{s}^{\epsilon,\mu^{\epsilon}})}{\epsilon}(t-s)}\mathbb{E}_{\bar{X}_{s}^{\epsilon,\mu^{\epsilon}}}[K(X_{s}^{\epsilon,\mu^{\epsilon}}-\bar{X}_{s}^{\epsilon,\mu^{\epsilon}})|\mathcal{F}_{s}^{B}]\psi(X_{s}^{\epsilon,\mu^{\epsilon}})|\mathcal{F}_{s}^{B}]\nonumber\\
&&-\mathbb{E}_{X_{t_{k}}^{\epsilon,\mu^{\epsilon}}}[e^{-\frac{\gamma(X_{t_{k}}^{\epsilon,\mu^{\epsilon}})}{\epsilon}(t-s)}\mathbb{E}_{\bar{X}_{t_{k}}^{\epsilon,\mu^{\epsilon}}}[K(X_{t_{k}}^{\epsilon,\mu^{\epsilon}}-\bar{X}_{t_{k}}^{\epsilon,\mu^{\epsilon}})|\mathcal{F}_{t_{k}}^{B}]\psi(X_{t_{k}}^{\epsilon,\mu^{\epsilon}})|\mathcal{F}_{t_{k}}^{B}]|\\
&=&\mathbb{E}|\mathbb{E}[e^{-\frac{\gamma(X_{s}^{\epsilon,\mu^{\epsilon}})}{\epsilon}(t-s)}\psi(X_{s}^{\epsilon,\mu^{\epsilon}})|\mathcal{F}_{s}^{B}]\cdot\mathbb{E}_{\bar{X}_{s}^{\epsilon,\mu^{\epsilon}}}[K(X_{s}^{\epsilon,\mu^{\epsilon}}-\bar{X}_{s}^{\epsilon,\mu^{\epsilon}})|\mathcal{F}_{s}^{B}]\\
&&-\mathbb{E}[e^{-\frac{\gamma(X_{t_{k}}^{\epsilon,\mu^{\epsilon}})}{\epsilon}(t-s)}\psi(X_{t_{k}}^{\epsilon,\mu^{\epsilon}})|\mathcal{F}_{s}^{B}]\cdot\mathbb{E}_{\bar{X}_{t_{k}}^{\epsilon,\mu^{\epsilon}}}[K(X_{t_{k}}^{\epsilon,\mu^{\epsilon}}-\bar{X}_{t_{k}}^{\epsilon,\mu^{\epsilon}})|\mathcal{F}_{s}^{B}]|\\
&\leq&\mathbb{E}|\mathbb{E}[(e^{-\frac{\gamma(X_{s}^{\epsilon,\mu^{\epsilon}})}{\epsilon}(t-s)}-e^{-\frac{\gamma(X_{t_{k}}^{\epsilon,\mu^{\epsilon}})}{\epsilon}(t-s)})\psi(X_{s}^{\epsilon,\mu^{\epsilon}})|\mathcal{F}_{s}^{B}]\mathbb{E}_{\bar{X}_{s}^{\epsilon,\mu^{\epsilon}}}[K(X_{s}^{\epsilon,\mu^{\epsilon}}-\bar{X}_{s}^{\epsilon,\mu^{\epsilon}})|\mathcal{F}_{s}^{B}]|\nonumber\\
&&+\mathbb{E}|\mathbb{E}[e^{-\frac{\gamma(X_{t_{k}}^{\epsilon,\mu^{\epsilon}})}{\epsilon}(t-s)}(\psi(X_{s}^{\epsilon,\mu^{\epsilon}})-\psi(X_{t_{k}}^{\epsilon,\mu^{\epsilon}}))|\mathcal{F}_{s}^{B}]\mathbb{E}_{\bar{X}_{s}^{\epsilon,\mu^{\epsilon}}}[K(X_{s}^{\epsilon,\mu^{\epsilon}}-\bar{X}_{s}^{\epsilon,\mu^{\epsilon}})|\mathcal{F}_{s}^{B}]|\nonumber\\
&&+\mathbb{E}|\mathbb{E}[e^{-\frac{\gamma(X_{t_{k}}^{\epsilon,\mu^{\epsilon}})}{\epsilon}(t-s)}\psi(X_{s}^{\epsilon,\mu^{\epsilon}})|\mathcal{F}_{s}^{B}]\mathbb{E}[K(x-\bar{X}_{s}^{\epsilon,\mu^{\epsilon}})-K(z-\bar{X}_{t_{k}}^{\epsilon,\mu^{\epsilon}})|\mathcal{F}_{s}^{B}]\Big|_{x=X_{s}^{\epsilon,\mu^{\epsilon}}, z=X_{t_{k}}^{\epsilon,\mu^{\epsilon}}}|\nonumber\\
&\leq&\frac{4TL_{\gamma}}{\epsilon}\mathbb{E}[\mathbb{E}[\|X_{s}^{\epsilon,\mu^{\epsilon}}-X_{t_{k}}^{\epsilon,\mu^{\epsilon}}\|\|\psi(X_{s}^{\epsilon,\mu^{\epsilon}})\||\mathcal{F}_{s}^{B}]\mathbb{E}_{\bar{X}_{s}^{\epsilon,\mu^{\epsilon}}}[\|K(X_{s}^{\epsilon,\mu^{\epsilon}}-\bar{X}_{s}^{\epsilon,\mu^{\epsilon}})\||\mathcal{F}_{s}^{B}]]\nonumber\\
&&+\|\psi\|_{Lip}\mathbb{E}[\mathbb{E}[\|X_{s}^{\epsilon,\mu^{\epsilon}}-X_{t_{k}}^{\epsilon,\mu^{\epsilon}}\||\mathcal{F}_{s}^{B}]\mathbb{E}_{\bar{X}_{s}^{\epsilon,\mu^{\epsilon}}}[\|K(X_{s}^{\epsilon,\mu^{\epsilon}}-\bar{X}_{s}^{\epsilon,\mu^{\epsilon}})\||\mathcal{F}_{s}^{B}]]\\
&&+L_{K}\mathbb{E}[\mathbb{E}[\|\psi(X_{s}^{\epsilon,\mu^{\epsilon}})\||\mathcal{F}_{s}^{B}]\mathbb{E}[\|(x-\bar{X}_{s}^{\epsilon,\mu^{\epsilon}})-(z-\bar{X}_{t_{k}}^{\epsilon,\mu^{\epsilon}})\||\mathcal{F}_{s}^{B}]\Big|_{x=X_{s}^{\epsilon,\mu^{\epsilon}}, z=X_{t_{k}}^{\epsilon,\mu^{\epsilon}}}]\\
&\leq&\frac{4TL_{\gamma}}{\epsilon}\|\psi\|_{Lip}(\mathbb{E}\|X_{s}^{\epsilon,\mu^{\epsilon}}-X_{t_{k}}^{\epsilon,\mu^{\epsilon}}\|^{2})^{\frac{1}{2}}(\mathbb{E}\|K(X_{s}^{\epsilon,\mu^{\epsilon}}-\bar{X}_{s}^{\epsilon,\mu^{\epsilon}})\|^{2})^{\frac{1}{2}}\\
&&+\|\psi\|_{Lip}(\mathbb{E}\|X_{s}^{\epsilon,\mu^{\epsilon}}-X_{t_{k}}^{\epsilon,\mu^{\epsilon}}\|^{2})^{\frac{1}{2}}(\mathbb{E}\|K(X_{s}^{\epsilon,\mu^{\epsilon}}-\bar{X}_{s}^{\epsilon,\mu^{\epsilon}})\|^{2})^{\frac{1}{2}}\\
&&+L_{K}\|\psi\|_{Lip}(\mathbb{E}\|(X_{s}^{\epsilon,\mu^{\epsilon}}-X_{t_{k}}^{\epsilon,\mu^{\epsilon}})-(\bar{X}_{s}^{\epsilon,\mu^{\epsilon}}-\bar{X}_{t_{k}}^{\epsilon,\mu^{\epsilon}})\|^{2})^{\frac{1}{2}}\\
&\leq& \frac{C_{T,X_{0},V_{0},L_{\gamma}, L_{K},\|\psi\|_{\infty}}}{\epsilon}\delta^{\frac{1}{2}}\Big(1+\frac{1}{\epsilon}\Big)\|\psi\|_{Lip},
\end{eqnarray*}
where $\bar{X}_{s}^{\epsilon,\mu^{\epsilon}}$ is an independent copy of $X_{s}^{\epsilon,\mu^{\epsilon}}$, $\mathbb{E}_{X}$ denote the expectation with respect to the distribution of $X$ and Remark $1$ in~\cite{T} is used in the second equality. Then by Lemma \ref{SUB},
\begin{eqnarray}\label{Pi1}
\frac{1}{\epsilon}|\Pi_{1}^{\epsilon}(t)|\leq C_{T,X_{0},V_{0},L_{\gamma}, L_{K}}\delta^{\frac{3}{2}}\Big(\frac{1}{\epsilon}+\frac{1}{\epsilon^{2}}\Big)\|\psi\|_{Lip}.
\end{eqnarray}
Similarly, direct computation yields
\begin{eqnarray}\label{Pi2Z}
&&\frac{1}{\epsilon}\Big\|\sum_{l=1}^{\infty}\int_{t_{k}}^{t}\langle e^{-\frac{\gamma(x)}{\epsilon}(t-s)}\sigma_{l}(x)(\rho_{s}^{\epsilon}(x)-\rho_{t_{k}}^{\epsilon}(x)),\psi\rangle dB_{s}^{l}\Big\|\nonumber\\
&\leq& \frac{1}{\epsilon}\Big\|\sum_{l=1}^{\infty}\int_{t_{k}}^{t}\langle e^{-\frac{\gamma(x)}{\epsilon}(t-s)}\sigma_{l}(x)\rho_{s}^{\epsilon}(x),\psi\rangle dB_{s}^{l}\Big\|+\frac{1}{\epsilon}\Big\|\sum_{l=1}^{\infty}\int_{t_{k}}^{t}\langle e^{-\frac{\gamma(x)}{\epsilon}(t-s)}\sigma_{l}(x)\rho_{t_{k}}^{\epsilon}(x),\psi\rangle dB_{s}^{l}\Big\|\nonumber
\end{eqnarray}
\begin{eqnarray}
&=&\frac{1}{\epsilon}\Big\|\sum_{l=1}^{\infty}\int_{t_{k}}^{t}\mathbb{E}[e^{-\frac{\gamma(X_{s}^{\epsilon,\mu^{\epsilon}})}{\epsilon}(t-s)}\sigma_{l}(X_{s}^{\epsilon,\mu^{\epsilon}})\psi(X_{s}^{\epsilon,\mu^{\epsilon}})|\mathcal{F}_{s}^{B}]dB_{s}^{l}\Big\|\nonumber\\
&&+\frac{1}{\epsilon}\Big\|\sum_{l=1}^{\infty}\int_{t_{k}}^{t}\mathbb{E}[e^{-\frac{\gamma(X_{t_{k}}^{\epsilon,\mu^{\epsilon}})}{\epsilon}(t-s)}\sigma_{l}(X_{t_{k}}^{\epsilon,\mu^{\epsilon}})\psi(X_{t_{k}}^{\epsilon,\mu^{\epsilon}})|\mathcal{F}_{s}^{B}]dB_{s}^{l}\Big\|\nonumber\\
&\triangleq&\Pi_{2,1}^{\epsilon}(t)+\Pi_{2,2}^{\epsilon}(t).
\end{eqnarray}
We only need to estimate $\Pi_{2,1}^{\epsilon}(t)$. The estimate for $\Pi_{2,2}^{\epsilon}(t)$ is exactly the same. In fact, by Lemma \ref{SUB} and $(\mathbf{H_{3}})$
\begin{eqnarray*}
\mathbb{E}\Pi_{2,1}^{\epsilon}(t)&=&\frac{1}{\epsilon}\mathbb{E}\Big\|\int_{t_{k}}^{t}\sum_{l=1}^{\infty}\mathbb{E}[e^{-\frac{\gamma(X_{s}^{\epsilon,\mu^{\epsilon}})}{\epsilon}(t-s)}\sigma_{l}(X_{s}^{\epsilon,\mu^{\epsilon}})\psi(X_{s}^{\epsilon,\mu^{\epsilon}})|\mathcal{F}_{s}^{B}]dB_{s}^{l}\Big\|\\
&\leq&\frac{C}{\epsilon}\Big(\sum_{l=1}^{\infty}\sup_{x}\|\sigma_{l}(x)\|^{2}\Big)^{\frac{1}{2}}\Big(\int_{t_{k}}^{t}(1+\mathbb{E}\|X_{s}^{\epsilon,\mu^{\epsilon}}\|^{2})ds\Big)^{\frac{1}{2}}\|\psi\|_{Lip}\\
&\leq&C_{T,X_{0},V_{0},L_{\sigma}}\frac{\sqrt{\delta}}{\epsilon}\|\psi\|_{Lip},
\end{eqnarray*}
then
\begin{eqnarray}\label{Pi2-1}
\sup_{0\leq t\leq T}\mathbb{E}\Pi_{2,1}^{\epsilon}(t)\leq C_{T,X_{0},V_{0},L_{\sigma}}\frac{\sqrt{\delta}}{\epsilon}\|\psi\|_{Lip}.
\end{eqnarray}
Combining (\ref{Pi2Z}) with (\ref{Pi2-1}), we have
\begin{eqnarray}\label{Pi2S}
\sup_{0\leq t\leq T}\mathbb{E}|\Pi_{2}^{\epsilon}(t)|\leq C_{T,X_{0},V_{0},L_{\sigma}}\frac{\sqrt{\delta}}{\epsilon}\|\psi\|_{Lip}.
\end{eqnarray}
Note that
\begin{eqnarray*}
&&\langle e^{-\frac{\gamma(x)}{\epsilon}(t-s)}\nabla_{x}[\rho_{s}^{\epsilon}\mathbb{E}^{x}(V_{s}^{\epsilon,\mu^{\epsilon}}\otimes V_{s}^{\epsilon,\mu^{\epsilon}})-\rho_{t_{k}}^{\epsilon}\mathbb{E}^{x}(V_{t_{k}}^{\epsilon,\mu^{\epsilon}}\otimes V_{t_{k}}^{\epsilon,\mu^{\epsilon}})],\psi\rangle\\
&=&-\langle[\rho_{s}^{\epsilon}\mathbb{E}^{x}(V_{s}^{\epsilon,\mu^{\epsilon}}\otimes V_{s}^{\epsilon,\mu^{\epsilon}})-\rho_{t_{k}}^{\epsilon}\mathbb{E}^{x}(V_{t_{k}}^{\epsilon,\mu^{\epsilon}}\otimes V_{t_{k}}^{\epsilon,\mu^{\epsilon}})], \nabla_{x}(e^{-\frac{\gamma(x)}{\epsilon}(t-s)}\psi)\rangle\\
&=&\langle \rho_{s}^{\epsilon}\mathbb{E}^{x}(V_{s}^{\epsilon,\mu^{\epsilon}}\otimes V_{s}^{\epsilon,\mu^{\epsilon}})-\rho_{t_{k}}^{\epsilon}\mathbb{E}^{x}(V_{t_{k}}^{\epsilon,\mu^{\epsilon}}\otimes V_{t_{k}}^{\epsilon,\mu^{\epsilon}}),e^{-\frac{\gamma(x)}{\epsilon}(t-s)}\frac{\nabla_{x}\gamma(x)}{\epsilon}(t-s)\psi\rangle\\
&&+\langle \rho_{s}^{\epsilon}\mathbb{E}^{x}(V_{s}^{\epsilon,\mu^{\epsilon}}\otimes V_{s}^{\epsilon,\mu^{\epsilon}})-\rho_{t_{k}}^{\epsilon}\mathbb{E}^{x}(V_{t_{k}}^{\epsilon,\mu^{\epsilon}}\otimes V_{t_{k}}^{\epsilon,\mu^{\epsilon}}),e^{-\frac{\gamma(x)}{\epsilon}(t-s)}\nabla_{x}\psi\rangle,
\end{eqnarray*}
and by $(\mathbf{H_{2}})$,
\begin{eqnarray*}
&&\frac{1}{\epsilon}|\langle \rho_{s}^{\epsilon}\mathbb{E}^{x}(\epsilon V_{s}^{\epsilon,\mu^{\epsilon}}\otimes V_{s}^{\epsilon,\mu^{\epsilon}}), \frac{t-s}{\epsilon}e^{-\frac{\gamma(x)}{\epsilon}(t-s)}\nabla_{x}\gamma(x)\psi\rangle|\\
&=&\frac{t-s}{\epsilon^{2}}|\mathbb{E}[\mathbb{E}^{X_{s}^{\epsilon,\mu^{\epsilon}}}(\epsilon V_{s}^{\epsilon,\mu^{\epsilon}}\otimes V_{s}^{\epsilon,\mu^{\epsilon}})e^{-\frac{\gamma(X_{s}^{\epsilon,\mu^{\epsilon}})}{\epsilon}(t-s)}\nabla_{x}\gamma(X_{s}^{\epsilon,\mu^{\epsilon}})\psi(X_{s}^{\epsilon,\mu^{\epsilon}})|\mathcal{F}_{s}^{B}]|\\
&\leq&\frac{C_{L_{\gamma},T}}{\epsilon^{2}}\|\psi\|_{Lip}\mathbb{E}[\|\mathbb{E}^{X_{s}^{\epsilon,\mu^{\epsilon}}}(\epsilon V_{s}^{\epsilon,\mu^{\epsilon}}\otimes V_{s}^{\epsilon,\mu^{\epsilon}})\||\mathcal{F}_{s}^{B}]\\
&\leq&\frac{C_{L_{\gamma},T}}{\epsilon^{2}}\mathbb{E}[\|\mathbb{E}^{X_{s}^{\epsilon,\mu^{\epsilon}}}(\epsilon V_{s}^{\epsilon,\mu^{\epsilon}}\otimes V_{s}^{\epsilon,\mu^{\epsilon}})\||\mathcal{F}_{s}^{B}]\|\psi\|_{Lip}.
\end{eqnarray*}
Similarly, 
\begin{eqnarray*}
&&\frac{1}{\epsilon}|\langle\rho_{s}^{\epsilon}\mathbb{E}^{x}(\epsilon V_{s}^{\epsilon,\mu^{\epsilon}}\otimes V_{s}^{\epsilon,\mu^{\epsilon}}),e^{-\frac{\gamma(x)}{\epsilon}(t-s)}\nabla_{x}\psi\rangle|\\
&=&\frac{1}{\epsilon}|\mathbb{E}[e^{-\frac{\gamma(X_{s}^{\epsilon,\mu^{\epsilon}})}{\epsilon}(t-s)}\mathbb{E}^{X_{s}^{\epsilon,\mu^{\epsilon}}}(\epsilon V_{s}^{\epsilon,\mu^{\epsilon}}\otimes V_{s}^{\epsilon,\mu^{\epsilon}})\nabla_{x}\psi(X_{s}^{\epsilon,\mu^{\epsilon}})|\mathcal{F}_{s}^{B}]|\\
&\leq&\frac{\|\psi\|_{Lip}}{\epsilon}\mathbb{E}[\|\mathbb{E}^{X_{s}^{\epsilon,\mu^{\epsilon}}}(\epsilon V_{s}^{\epsilon,\mu^{\epsilon}}\otimes V_{s}^{\epsilon,\mu^{\epsilon}})\||\mathcal{F}_{s}^{B}].
\end{eqnarray*}
Then
\begin{eqnarray}\label{Pi3}
\mathbb{E}|\Pi_{3}^{\epsilon}(t)|&\leq&\mathbb{E}\int_{t_{k}}^{t}\Big[\frac{C_{L_{\gamma},T}}{\epsilon^{2}}\mathbb{E}[\|\mathbb{E}^{X_{s}^{\epsilon,\mu^{\epsilon}}}(\epsilon V_{s}^{\epsilon,\mu^{\epsilon}}\otimes V_{s}^{\epsilon,\mu^{\epsilon}})\||\mathcal{F}_{s}^{B}]\|\psi\|_{Lip}\nonumber\\
&&+\frac{\|\psi\|_{Lip}}{\epsilon}\mathbb{E}[\|\mathbb{E}^{X_{s}^{\epsilon,\mu^{\epsilon}}}(\epsilon V_{s}^{\epsilon,\mu^{\epsilon}}\otimes V_{s}^{\epsilon,\mu^{\epsilon}})\||\mathcal{F}_{s}^{B}]\Big]ds\nonumber\\
&\leq&\frac{C_{L_{\gamma},T}}{\epsilon^{2}}\|\psi\|_{Lip}\int_{t_{k}}^{t}\mathbb{E}\|\mathbb{E}^{X_{s}^{\epsilon,\mu^{\epsilon}}}(\epsilon V_{s}^{\epsilon,\mu^{\epsilon}}\otimes V_{s}^{\epsilon,\mu^{\epsilon}})\|ds\nonumber\\
&&+\frac{\|\psi\|_{Lip}}{\epsilon}\int_{t_{k}}^{t}\mathbb{E}\|\mathbb{E}^{X_{s}^{\epsilon,\mu^{\epsilon}}}(\epsilon V_{s}^{\epsilon,\mu^{\epsilon}}\otimes V_{s}^{\epsilon,\mu^{\epsilon}})\|ds.
\end{eqnarray}
Combining (\ref{ZZ}) with (\ref{Pi1}), (\ref{Pi2S}) and (\ref{Pi3}), we have
\begin{eqnarray*}\label{ZZ1}
\sup_{0\leq t\leq T}\mathbb{E}|\langle Z_{t}^{\epsilon},\psi\rangle|&\leq& C_{T,X_{0},V_{0},L_{\gamma}, L_{K}}\delta^{\frac{3}{2}}\Big(\frac{1}{\epsilon}+\frac{1}{\epsilon^{2}}\Big)\|\psi\|_{Lip}+C_{T,X_{0},V_{0},L_{\sigma}}\frac{\sqrt{\delta}}{\epsilon}\|\psi\|_{Lip}\nonumber\\
&&+C_{L_{\gamma},T}\Big(\frac{1}{\epsilon}+\frac{1}{\epsilon^{2}}\Big)\|\psi\|_{Lip}\int_{t_{k}}^{t}\mathbb{E}\|\mathbb{E}^{X_{s}^{\epsilon,\mu^{\epsilon}}}(\epsilon V_{s}^{\epsilon,\mu^{\epsilon}}\otimes V_{s}^{\epsilon,\mu^{\epsilon}})\|ds.\nonumber\\
&\leq&C_{T,X_{0},V_{0},L_{\gamma}, L_{K},}\delta^{\frac{3}{2}}\Big(\frac{1}{\epsilon}+\frac{1}{\epsilon^{2}}\Big)\|\psi\|_{Lip}+C_{T,X_{0},V_{0},L_{\sigma}}\frac{\sqrt{\delta}}{\epsilon}\|\psi\|_{Lip}\nonumber\\
&&+C_{T,X_{0},V_{0},L_{\gamma},M}\delta\Big(\frac{1}{\epsilon}+\frac{1}{\epsilon^{2}}\Big)\|\psi\|_{Lip}\nonumber\\
&\leq&C_{T,X_{0},V_{0},L_{\gamma},L_{\sigma},L_{K},M}\Big(\frac{\sqrt{\delta}}{\epsilon}+\frac{\delta}{\epsilon}+\frac{\delta}{\epsilon^{2}}\Big)\|\psi\|_{Lip}.
\end{eqnarray*}
\end{proof}
Next we determine the limit of $\hat{Y}_{t}^{\epsilon}$ as $\epsilon\to0$. For this, we introduce the following equation
\begin{eqnarray}\label{FU2}
d\tilde{Y}_{t}^{\epsilon}&=&-\frac{\gamma(x)}{\epsilon}\tilde{Y}_{t}^{\epsilon}dt+\frac{K*\rho(x)}{\epsilon}\rho(x)dt-\nabla_{x}\cdot[\rho(x)\mathbb{E}^{x}(V_{\tau}^{\epsilon,\mu^{\epsilon}}\otimes V_{\tau}^{\epsilon,\mu^{\epsilon}})]dt\nonumber\\
&&+\frac{1}{\epsilon}\sum_{k=1}^{\infty}\sigma_{k}(x)dB_{t}^{k}\rho(x),
\end{eqnarray}
where $\rho\in \mathcal{P}_{G}$, which is the space consisting of Gaussian measure, $\tau>0$. We firstly give the following result.
\begin{lemma}\label{YS}
Suppose that $(\mathbf{H_{1}})$-$(\mathbf{H_{4}})$ hold, then for every $T>0$ and $\psi\in C_{0}^{\infty}(\mathbb{R}^{d};\mathbb{R}^{d})$, there exists a constant $C_{T}>0$ such that 
\begin{eqnarray*}
\sup_{0\leq t\leq T}\mathbb{E}|\langle Y_{t}^{\epsilon},\psi\rangle|\leq \frac{C_{T,X_{0},V_{0},L_{K},\gamma_{0},L_{\sigma},M}}{\sqrt{\epsilon}}\|\psi\|_{Lip}.
\end{eqnarray*}
\end{lemma}
\begin{proof}
By (\ref{Macro2}), we have
\begin{eqnarray}\label{YP}
d\langle Y_{t}^{\epsilon},\psi\rangle&=&-\Big\langle \frac{\gamma(x)}{\epsilon}Y_{t}^{\epsilon},\psi\Big\rangle dt+\Big\langle\frac{K*\rho_{t}^{\epsilon}(x)}{\epsilon}\rho_{t}^{\epsilon},\psi\Big\rangle dt-\langle \nabla_{x}\cdot[\rho_{t}^{\epsilon}\mathbb{E}^{x}(V_{t}^{\epsilon,\mu^{\epsilon}}\otimes V_{t}^{\epsilon,\mu^{\epsilon}})],\psi\rangle dt\nonumber\\
&&+\frac{1}{\epsilon}\sum_{k=1}^{\infty}\langle\sigma_{k}(x)\rho_{t}^{\epsilon},\psi\rangle dB_{t}^{k}.
\end{eqnarray}
Then by It${\rm \hat{o}}$'s formula, we have
\begin{eqnarray*}
\frac{1}{2}\frac{d}{dt}|\langle Y_{t}^{\epsilon},\psi\rangle|^{2}&=&\langle Y_{t}^{\epsilon},\psi\rangle\Big[-\Big\langle \frac{\gamma(x)}{\epsilon}Y_{t}^{\epsilon},\psi\Big\rangle+\Big\langle\frac{K*\rho_{t}^{\epsilon}(x)}{\epsilon}\rho_{t}^{\epsilon},\psi\Big\rangle\\
&&-\langle \nabla_{x}\cdot[\rho_{t}^{\epsilon}\mathbb{E}^{x}(V_{t}^{\epsilon,\mu^{\epsilon}}\otimes V_{t}^{\epsilon,\mu^{\epsilon}}],\psi\rangle+\frac{1}{\epsilon}\sum_{k=1}^{\infty}\langle\sigma_{k}(x)\rho_{t}^{\epsilon},\psi\rangle dB_{t}^{k}\Big]\\
&&+\frac{1}{2\epsilon^{2}}\sum_{k=1}^{\infty}|\langle\sigma_{k}(x)\rho_{t}^{\epsilon},\psi\rangle|^{2}\\
&\leq&-\frac{\gamma_{0}}{\epsilon}|\langle Y_{t}^{\epsilon},\psi\rangle|^{2}+\frac{1}{\epsilon}|\langle Y_{t}^{\epsilon},\psi\rangle||\langle (K*\rho_{t}^{\epsilon})\rho_{t}^{\epsilon},\psi\rangle|\\
&&+\frac{1}{\epsilon}|\langle Y_{t}^{\epsilon},\psi\rangle||\langle \rho_{t}^{\epsilon}\mathbb{E}^{x}(\epsilon V_{t}^{\epsilon,\mu^{\epsilon}}\otimes V_{t}^{\epsilon,\mu^{\epsilon}}),\nabla_{x}\psi\rangle|\\
&&+\frac{1}{\epsilon}\langle Y_{t}^{\epsilon},\psi\rangle\sum_{k=1}^{\infty}\langle \sigma_{k}(x)\rho_{t}^{\epsilon},\psi\rangle dB_{t}^{k}+\frac{1}{2\epsilon^{2}}\sum_{l=1}^{\infty}|\langle\sigma_{k}(x)\rho_{t}^{\epsilon},\psi\rangle|^{2}.
\end{eqnarray*}
Taking expectation, utilizing the Young inequality, Lemma \ref{SUB} and the definition of the expectation, we obtain
\begin{eqnarray*}
&&\frac{1}{2}\frac{d}{dt}\mathbb{E}|\langle Y_{t}^{\epsilon},\psi\rangle|^{2}\leq-\frac{\gamma_{0}}{2\epsilon}\mathbb{E}|\langle Y_{t}^{\epsilon},\psi\rangle|^{2}+\frac{1}{\gamma_{0}\epsilon}\mathbb{E}|\langle\mathbb{E}_{\bar{X}_{t}^{\epsilon,\mu^{\epsilon}}}[K(x-\bar{X}_{t}^{\epsilon,\mu^{\epsilon}})|\mathcal{F}_{t}^{B}]\rho_{t}^{\epsilon},\psi\rangle|^{2}\\
&&+\frac{1}{\gamma_{0}\epsilon}\mathbb{E}\|\mathbb{E}[\mathbb{E}^{X_{t}^{\epsilon,\mu^{\epsilon}}}(\epsilon V_{t}^{\epsilon,\mu^{\epsilon}}\otimes V_{t}^{\epsilon,\mu^{\epsilon}})\nabla_{x}\psi(X_{t}^{\epsilon,\mu^{\epsilon}})|\mathcal{F}_{t}^{B}]\|^{2}+\frac{1}{2\epsilon^{2}}\sum_{l=1}^{\infty}\sup_{x}\|\sigma_{l}(x)\|^{2}\mathbb{E}\|\psi(X_{t}^{\epsilon,\mu^{\epsilon}})\|^{2}\\
&\leq&-\frac{\gamma_{0}}{2\epsilon}\mathbb{E}|\langle Y_{t}^{\epsilon},\psi\rangle|^{2}+\frac{1}{\gamma_{0}\epsilon}\mathbb{E}\|\mathbb{E}_{X_{t}^{\epsilon,\mu^{\epsilon}}}[\mathbb{E}_{\bar{X}_{t}^{\epsilon,\mu^{\epsilon}}}[K(X_{t}^{\epsilon,\mu^{\epsilon}}-\bar{X}_{t}^{\epsilon,\mu^{\epsilon}})|\mathcal{F}_{t}^{B}]\psi(X_{t}^{\epsilon,\mu^{\epsilon}})|\mathcal{F}_{t}^{B}]\|^{2}\\
&&+\frac{1}{\gamma_{0}\epsilon}\mathbb{E}\|\mathbb{E}[\mathbb{E}^{X_{t}^{\epsilon,\mu^{\epsilon}}}(\epsilon V_{t}^{\epsilon,\mu^{\epsilon}}\otimes V_{t}^{\epsilon,\mu^{\epsilon}})\nabla_{x}\psi(X_{t}^{\epsilon,\mu^{\epsilon}})|\mathcal{F}_{t}^{B}]\|^{2}+\frac{\|\psi\|_{Lip}}{\epsilon^{2}}\sum_{l=1}^{\infty}\sup_{x}\|\sigma_{l}(x)\|^{2}(1+\mathbb{E}\|X_{t}^{\epsilon,\mu^{\epsilon}}\|^{2})\\
&\leq&-\frac{\gamma_{0}}{2\epsilon}\mathbb{E}|\langle Y_{t}^{\epsilon},\psi\rangle|^{2}+\frac{1}{\gamma_{0}\epsilon}\|\psi\|_{Lip}^{2}\mathbb{E}_{X_{t}^{\epsilon,\mu^{\epsilon}}}[\mathbb{E}_{\bar{X}_{t}^{\epsilon,\mu^{\epsilon}}}\|K(X_{t}^{\epsilon,\mu^{\epsilon}}-\bar{X}_{t}^{\epsilon,\mu^{\epsilon}})\|^{2}]
\end{eqnarray*}
\begin{eqnarray*}
&&+\frac{1}{\gamma_{0}\epsilon}\|\psi\|_{Lip}^{2}\mathbb{E}\Big\|\frac{1}{2}\sum_{k=1}^{\infty}\|\sigma_{k}(X_{t}^{\epsilon,\mu^{\epsilon}})\|^{2}\gamma(X_{t}^{\epsilon,\mu^{\epsilon}})^{-1}\cdot I_{d\times d}+\epsilon C(X_{t}^{\epsilon,\mu^{\epsilon}},t)\Big\|^{2}\\
&&+\frac{\|\psi\|_{Lip}}{\epsilon^{2}}\sum_{k=1}^{\infty}\sup_{x}\|\sigma_{k}(x)\|^{2}(1+\mathbb{E}\|X_{t}^{\epsilon,\mu^{\epsilon}}\|^{2})\\
&\leq&-\frac{\gamma_{0}}{2\epsilon}\mathbb{E}|\langle Y_{t}^{\epsilon},\psi\rangle|^{2}+\frac{C_{\gamma_{0},L_{K}}}{\epsilon}\mathbb{E}_{X_{t}^{\epsilon,\mu^{\epsilon}}}[\mathbb{E}_{\bar{X}_{t}^{\epsilon,\mu^{\epsilon}}}(1+\|X_{t}^{\epsilon,\mu^{\epsilon}}\|^{2}+\|\bar{X}_{t}^{\epsilon,\mu^{\epsilon}}\|^{2})]\\
&&+\frac{C_{\gamma_{0}}}{\epsilon}\|\psi\|_{Lip}^{2}\Big(\sup_{x}\sum_{l=1}^{\infty}\|\sigma_{l}(x)\|^{2}\Big)+C_{\gamma_{0},M}\|\psi\|_{Lip}^{2}(1+\mathbb{E}\|X_{t}^{\epsilon,\mu^{\epsilon}}\|^{2})\\
&&+\frac{C_{L_{\sigma}}}{\epsilon^{2}}\|\psi\|_{Lip}(1+\mathbb{E}\|X_{t}^{\epsilon,\mu^{\epsilon}}\|^{2})\\
&\leq&-\frac{\gamma_{0}}{2\epsilon}\mathbb{E}|\langle Y_{t}^{\epsilon},\psi\rangle|^{2}+C_{T,X_{0},V_{0},L_{K},\gamma_{0},L_{\sigma},M}\Big(\frac{1}{\epsilon}+\frac{1}{\epsilon^{2}}\Big)\|\psi\|_{Lip}^{2},
\end{eqnarray*}
then Gronwall's inequality yields
\begin{eqnarray*}
\sup_{0\leq t\leq T}\mathbb{E}|\langle Y_{t}^{\epsilon},\psi\rangle|^{2}\leq \frac{C_{T,X_{0},V_{0},L_{K},\gamma_{0},L_{\sigma},M}}{\epsilon}\|\psi\|_{Lip}^{2}.
\end{eqnarray*}
\end{proof}

Then for any $t\in[t_{k},t_{k+1}], k\in\mathbb{N}$, we have the following estimate.
\begin{lemma}\label{YR}
Suppose that $(\mathbf{H_{1}})$-$(\mathbf{H_{4}})$ hold, then for every $x\in \mathbb{R}^{d}$, fixed $\rho_{t}^{\epsilon}=\rho\in\mathcal{P}_{G}$ and each $0< s_{0}<\delta$, $\psi\in C_{0}^{\infty}(\mathbb{R}^{d};\mathbb{R}^{d})$, 
\begin{eqnarray*}
\mathbb{E}\Big|\int_{t_{k}}^{t}\langle \tilde{Y}_{s}^{\epsilon}-Y^{*}_{s},\psi\rangle ds\Big|&\leq&\sqrt{\epsilon}\delta C_{s_{0},\gamma_{0},M,L_{K},L_{\sigma},X_{0},V_{0}}\|\psi\|_{Lip}+\epsilon\delta C_{s_{0},\gamma_{0}}|\langle (K*\rho(x))\rho(x),\psi\rangle|\nonumber\\
&&+(2\epsilon\delta+\delta^{2})C_{M,\gamma_{0},L_{\gamma}}|\langle(1+\|x\|^{2})\rho(x),\psi\rangle|\nonumber\\
&&+\Big(\epsilon\delta+\frac{\delta^{2}}{\epsilon}\Big)C_{\gamma_{0},s_{0},L_{\sigma},L_{\gamma}}|\langle\rho(x),\psi\rangle|\nonumber\\
&&+\epsilon\delta C_{M,\gamma_{0}}|\langle (1+\|x\|^{2})\rho(x),\nabla_{x}\psi\rangle|+\epsilon\delta C_{\gamma_{0},s_{0},L_{\sigma}}|\langle \rho(x),\nabla_{x}\psi\rangle|\nonumber\\
&&+C_{L_{\sigma}}|\langle\rho,|\psi|\rangle|o(\sqrt{\epsilon} e^{-\frac{\gamma_{0}}{\epsilon}s_{0}}),
\end{eqnarray*}
where
\begin{eqnarray}\label{EEN}
Y^{*}_{t}(x)&=&\gamma(x)^{-1}(K*\rho(x))\rho(x)+\frac{1}{2}\sum_{k=1}^{\infty}\|\sigma_{k}(x)\|^{2}\nabla_{x}\gamma(x)\gamma(x)^{-3}\rho(x)\\
&&+\frac{1}{2}\nabla_{x}\cdot\Big(\sum_{k=1}^{\infty}\|\sigma_{k}(x)\|^{2}\rho(x)\gamma(x)^{-2}I_{d\times d}\Big)+\sum_{k=1}^{\infty}\sigma_{k}(x)\rho(x)\gamma(x)^{-1}\dot{B}_{t}^{k}.\nonumber
\end{eqnarray}
\end{lemma}
\begin{proof}
By the Duhamel principle, for any $t\in[t_{k},t_{k+1}]$, we have
\begin{eqnarray*}
\tilde{Y}_{t}^{\epsilon}&=&e^{-\frac{\gamma(x)}{\epsilon}(t-t_{k})}Y_{t_{k}}^{\epsilon}+\frac{1}{\epsilon}\int_{t_{k}}^{t}e^{-\frac{\gamma(x)}{\epsilon}(t-s)}(K*\rho(x))\rho(x)ds\\
&&-\int_{t_{k}}^{t}e^{-\frac{\gamma(x)}{\epsilon}(t-s)}\nabla_{x}\cdot[\rho(x)\mathbb{E}^{x}(V_{\tau}^{\epsilon,\mu^{\epsilon}}\otimes V_{\tau}^{\epsilon,\mu^{\epsilon}})]ds+\frac{1}{\epsilon}\sum_{l=1}^{\infty}\int_{t_{k}}^{t}e^{-\frac{\gamma(x)}{\epsilon}(t-s)}\sigma_{l}(x)dB_{s}^{l}\rho(x),
\end{eqnarray*}
then
\begin{eqnarray}\label{YPZ}
\langle\tilde{Y}_{t}^{\epsilon},\psi\rangle&=&\langle e^{-\frac{\gamma(x)}{\epsilon}(t-t_{k})}Y_{t_{k}}^{\epsilon},\psi\rangle+\frac{1}{\epsilon}\int_{t_{k}}^{t}\langle e^{-\frac{\gamma(x)}{\epsilon}(t-s)}(K*\rho(x))\rho(x),\psi\rangle ds\nonumber\\
&&-\int_{t_{k}}^{t}\langle e^{-\frac{\gamma(x)}{\epsilon}(t-s)}\nabla_{x}\cdot[\rho(x)\mathbb{E}^{x}(V_{\tau}^{\epsilon,\mu^{\epsilon}}\otimes V_{\tau}^{\epsilon,\mu^{\epsilon}})],\psi\rangle ds\nonumber\\
&&+\frac{1}{\epsilon}\sum_{l=1}^{\infty}\int_{t_{k}}^{t}\langle e^{-\frac{\gamma(x)}{\epsilon}(t-s)}\sigma_{l}(x)\rho(x),\psi\rangle dB_{s}^{l}\nonumber\\
&=&\langle e^{-\frac{\gamma(x)}{\epsilon}(t-t_{k})}Y_{t_{k}}^{\epsilon},\psi\rangle+\frac{1}{\epsilon}\int_{t_{k}}^{t}\langle e^{-\frac{\gamma(x)}{\epsilon}(t-s)}(K*\rho(x))\rho(x),\psi\rangle ds\nonumber\\
&&+\int_{t_{k}}^{t}\langle \rho(x)\mathbb{E}^{x}(V_{\tau}^{\epsilon,\mu^{\epsilon}}\otimes V_{\tau}^{\epsilon,\mu^{\epsilon}}),e^{-\frac{\gamma(x)}{\epsilon}(t-s)}\frac{t-s}{\epsilon}\nabla_{x}\gamma(x)\psi\rangle ds\nonumber\\
&&-\int_{t_{k}}^{t}\langle \rho(x)\mathbb{E}^{x}(V_{\tau}^{\epsilon,\mu^{\epsilon}}\otimes V_{\tau}^{\epsilon,\mu^{\epsilon}}),e^{-\frac{\gamma(x)}{\epsilon}(t-s)}\nabla_{x}\psi\rangle ds\nonumber\\
&&+\frac{1}{\epsilon}\sum_{l=1}^{\infty}\int_{t_{k}}^{t}\langle e^{-\frac{\gamma(x)}{\epsilon}(t-s)}\sigma_{l}(x)\rho(x),\psi\rangle dB_{s}^{l}.
\end{eqnarray}
Firstly, for any $s_{0}>0$, $s_{0}\leq t\leq T$ and the fact $\sup_{x\geq0}x^{p}e^{-ax}\leq C_{p}a^{-p}$, for any $a>0$, we obtain
\begin{eqnarray}\label{YP1}
\langle e^{-\frac{\gamma(x)}{\epsilon}(t-t_{k})}Y_{t_{k}}^{\epsilon},\psi\rangle&=&\Big\langle \Big(\frac{\epsilon}{t-t_{k}}\Big)\Big(\frac{t-t_{k}}{\epsilon}\Big)e^{-\frac{\gamma(x)}{\epsilon}(t-t_{k})}Y_{t_{k}}^{\epsilon},\psi\Big\rangle\nonumber\\
&\triangleq&\epsilon\langle C_{1}(x,t),\psi\rangle,
\end{eqnarray}
where $C_{1}(x,t)$ satisfies
\begin{eqnarray}\label{C1}
|\langle C_{1}(x,t),\psi\rangle|&\leq& \frac{1}{t-t_{k}}\frac{t-t_{k}}{\epsilon}e^{-\frac{\gamma_{0}}{\epsilon}(t-t_{k})}|\langle Y_{t_{k}}^{\epsilon},\psi\rangle|\nonumber\\
&\leq&C_{s_{0},\gamma_{0}}|\langle Y_{t_{k}}^{\epsilon},\psi\rangle|.
\end{eqnarray}
Direct calculations yield
\begin{eqnarray}\label{YP2}
&&\frac{1}{\epsilon}\int_{t_{k}}^{t}\langle e^{-\frac{\gamma(x)}{\epsilon}(t-s)}(K*\rho(x))\rho(x),\psi\rangle ds\nonumber\\
&=&\langle \gamma(x)^{-1}(1-e^{-\frac{\gamma(x)}{\epsilon}(t-t_{k})})(K*\rho(x))\rho(x),\psi\rangle \nonumber\\
&\triangleq&\langle\gamma(x)^{-1}(K*\rho(x))\rho(x),\psi\rangle+\langle C_{2}(x,t,\epsilon),\psi\rangle,
\end{eqnarray}
where $C_{2}(x,t,\epsilon)$ satisfies
\begin{eqnarray}\label{C2}
|\langle C_{2}(x,t,\epsilon,\rho),\psi\rangle|&=&|\langle \gamma(x)^{-1}e^{-\frac{\gamma(x)}{\epsilon}(t-s)}(K*\rho(x))\rho(x),\psi\rangle|\nonumber\\
&\leq&e^{-\frac{\gamma_{0}}{\epsilon}(t-s)}|\langle \gamma(x)^{-1}(K*\rho(x))\rho(x),\psi\rangle|\nonumber\\
&\leq&\epsilon C_{s_{0},\gamma_{0}}|\langle (K*\rho(x))\rho(x),\psi\rangle|.
\end{eqnarray}
Furthermore, let
\begin{eqnarray*}
Z_{1}^{\epsilon}(t,x)&=&\int_{t_{k}}^{t}\nabla_{x}\gamma(x)e^{-\frac{\gamma(x)}{\epsilon}(t-s)}\frac{t-s}{\epsilon}\mathbb{E}^{x}(V_{\tau}^{\epsilon,\mu^{\epsilon}}\otimes V_{\tau}^{\epsilon,\mu^{\epsilon}})\rho(x)ds\\
&=&\frac{1}{\epsilon^{2}}\nabla_{x}\gamma(x)\int_{t_{k}}^{t}(t-s)e^{-\frac{\gamma(x)}{\epsilon}(t-s)}\Big[\frac{1}{2}\sum_{k=1}^{\infty}\|\sigma_{k}(x)\|^{2}\gamma(x)^{-1}\cdot I_{d\times d}+\epsilon C(x,\tau)\Big]\rho(x)ds.
\end{eqnarray*}
Integration by parts twice yields
\begin{eqnarray}\label{FB}
&&\frac{1}{\epsilon^{2}}\int_{t_{k}}^{t}(t-s)e^{-\frac{\gamma(x)}{\epsilon}(t-s)}ds\nonumber\\
&=&-\gamma(x)^{-1}\frac{t-t_{k}}{\epsilon}e^{-\frac{\gamma(x)}{\epsilon}(t-t_{k})}+\gamma(x)^{-2}-\gamma(x)^{-2}e^{-\frac{\gamma(x)}{\epsilon}(t-t_{k})}\nonumber\\
&=&\gamma(x)^{-2}-\Big(\gamma(x)^{-1}\frac{t-t_{k}}{\epsilon}e^{-\frac{\gamma(x)}{\epsilon}(t-t_{k})}+\gamma(x)^{-2}e^{-\frac{\gamma(x)}{\epsilon}(t-t_{k})}\Big).
\end{eqnarray}
Then
\begin{eqnarray}\label{YP3}
Z_{1}^{\epsilon}(t,x)&=&\nabla_{x}\gamma(x)\Big[\gamma(x)^{-2}-\Big(\gamma(x)^{-1}\frac{t-t_{k}}{\epsilon}e^{-\frac{\gamma(x)}{\epsilon}(t-t_{k})}+\gamma(x)^{-2}e^{-\frac{\gamma(x)}{\epsilon}(t-t_{k})}\Big)\Big]\nonumber\\
&&\cdot\Big[\frac{1}{2}\sum_{l=1}^{\infty}\|\sigma_{l}(x)\|^{2}\gamma(x)^{-1}\cdot I_{d\times d}+\epsilon C(x,\tau)\Big]\rho(x)\nonumber\\
&=&\frac{1}{2}\sum_{l=1}^{\infty}\|\sigma_{l}(x)\|^{2}\nabla_{x}\gamma(x)\gamma(x)^{-3}\rho(x)+\Big[\epsilon\nabla_{x}\gamma(x)\gamma(x)^{-2}C(x,\tau)\rho(x)\nonumber\\
&&\quad-\nabla_{x}\gamma(x)\Big(\gamma(x)^{-1}\frac{t-t_{k}}{\epsilon}e^{-\frac{\gamma(x)}{\epsilon}(t-t_{k})}+\gamma(x)^{-2}e^{-\frac{\gamma(x)}{\epsilon}(t-t_{k})}\Big)\nonumber\\
&&\quad\cdot\Big(\frac{1}{2}\sum_{l=1}^{\infty}\|\sigma_{l}(x)\|^{2}\gamma(x)^{-1}\cdot I_{d\times d}+\epsilon C(x,\tau)\rho(x)\Big)\Big]\nonumber\\
&\triangleq&\frac{1}{2}\sum_{l=1}^{\infty}\|\sigma_{l}(x)\|^{2}\nabla_{x}\gamma(x)\gamma(x)^{-3}\rho(x)+C_{3}(x,\tau,t,\epsilon,\rho),
\end{eqnarray}
and
\begin{eqnarray}\label{YP3-1}
\langle Z_{1}^{\epsilon}(t,x),\psi\rangle=\Big\langle \frac{1}{2}\sum_{l=1}^{\infty}\|\sigma_{l}(x)\|^{2}\nabla_{x}\gamma(x)\gamma(x)^{-3}\rho(x),\psi\Big\rangle+\langle C_{3}(x,\tau,t,\epsilon,\rho),\psi\rangle,
\end{eqnarray}
where $C_{3}(x,\tau,t,\epsilon,\rho)$ satisfies
\begin{eqnarray}
|\langle C_{3}(x,\tau,t,\epsilon,\rho),\psi\rangle|&\leq&|\epsilon\langle\nabla_{x}\gamma(x)\gamma(x)^{-2}C(x,\tau)\rho(x),\psi\rangle|\nonumber\\
&&+\Big|\Big\langle \nabla_{x}\gamma(x)\gamma(x)^{-1}\frac{t-t_{k}}{\epsilon}e^{-\frac{\gamma(x)}{\epsilon}(t-t_{k})}\frac{1}{2}\sum_{l=1}^{\infty}\|\sigma_{l}(x)\|^{2}\gamma(x)^{-1}\rho(x),\psi\Big\rangle\Big|\nonumber
\end{eqnarray}
\begin{eqnarray}\label{C3}
&&+\Big|\Big\langle \nabla_{x}\gamma(x)\gamma(x)^{-1}\frac{t-t_{k}}{\epsilon}e^{-\frac{\gamma(x)}{\epsilon}(t-t_{k})}\epsilon C(x,\tau)\rho(x),\psi\Big\rangle\Big|\nonumber\\
&&+\Big|\Big\langle \nabla_{x}\gamma(x)\gamma(x)^{-2}e^{-\frac{\gamma(x)}{\epsilon}(t-t_{k})}\frac{1}{2}\sum_{l=1}^{\infty}\|\sigma_{l}(x)\|^{2}\gamma(x)^{-1}\rho(x),\psi\Big\rangle\Big|\nonumber\\
&&+|\langle \nabla_{x}\gamma(x)\gamma(x)^{-2}e^{-\frac{\gamma(x)}{\epsilon}(t-t_{k})}\epsilon C(x,\tau)\rho(x),\psi\rangle|\nonumber\\
&\leq&\epsilon |\langle\nabla_{x}\gamma(x)\gamma(x)^{-2}C(x,\tau)\rho(x),\psi\rangle|\nonumber\\
&&+\frac{1}{2}\frac{t-t_{k}}{\epsilon}e^{-\frac{\gamma_{0}}{\epsilon}(t-t_{k})}\sup_{x}\sum_{l=1}^{\infty}\|\sigma_{l}(x)\|^{2}|\langle \nabla_{x}\gamma(x)\gamma(x)^{-2}\rho(x),\psi\rangle|\nonumber\\
&&+\frac{t-t_{k}}{\epsilon}e^{-\frac{\gamma_{0}}{\epsilon}(t-t_{k})}\epsilon |\langle \nabla_{x}\gamma(x)\gamma(x)^{-2}\rho(x)C(x,\tau),\psi\rangle|\nonumber\\
&&+\frac{1}{2}\sup_{x}\sum_{l=1}^{\infty}\|\sigma_{l}(x)\|^{2}e^{-\frac{\gamma_{0}}{\epsilon}(t-t_{k})}|\langle \nabla_{x}\gamma(x)\gamma(x)^{-3}\rho(x),\psi\rangle|\nonumber\\
&&+\epsilon e^{-\frac{\gamma_{0}}{\epsilon}(t-t_{k})}|\langle \nabla_{x}\gamma(x)\gamma(x)^{-2}C(x,\tau)\rho(x),\psi\rangle|\nonumber\\
&\leq&(2\epsilon+\delta)C_{M,\gamma_{0},L_{\gamma}}|\langle(1+\|x\|^{2})\rho(x),\psi\rangle|+\Big(\epsilon+\frac{\delta}{\epsilon}\Big)C_{\gamma_{0},s_{0},L_{\sigma},L_{\gamma}}|\langle\rho(x),\psi\rangle|.\end{eqnarray}
Similarly, by Lemma \ref{LVV},
\begin{eqnarray*}
Z_{2}^{\epsilon}(t,x)&=&\int_{t_{k}}^{t}e^{-\frac{\gamma(x)}{\epsilon}(t-s)}\rho(x)\mathbb{E}^{x}(V_{\tau}^{\epsilon,\mu^{\epsilon}}\otimes V_{\tau}^{\epsilon,\mu^{\epsilon}})ds\\
&=&\gamma(x)^{-1}(1-e^{-\frac{\gamma(x)}{\epsilon}(t-t_{k})})\rho(x)\mathbb{E}^{x}(\epsilon V_{\tau}^{\epsilon,\mu^{\epsilon}}\otimes V_{\tau}^{\epsilon,\mu^{\epsilon}})\\
&=&\frac{1}{2}\sum_{l=1}^{\infty}\|\sigma_{l}(x)\|^{2}\rho(x)\gamma(x)^{-2}\cdot I_{d\times d}+\epsilon\gamma(x)^{-1}\rho(x)C(x,\tau)\\
&&-\gamma(x)^{-1}e^{-\frac{\gamma(x)}{\epsilon}(t-t_{k})}\Big[\frac{1}{2}\sum_{k=1}^{\infty}\|\sigma_{k}(x)\|^{2}\rho(x)\gamma(x)^{-1}\cdot I_{d\times d}+\epsilon\rho(x)C(x,\tau)\Big]\\
&\triangleq&\frac{1}{2}\sum_{k=1}^{\infty}\|\sigma_{k}(x)\|^{2}\rho(x)\gamma(x)^{-2}\cdot I_{d\times d}+C_{4}(x,\tau,t,\epsilon,\rho),
\end{eqnarray*}
and
\begin{eqnarray}\label{YP4}
\langle Z_{2}^{\epsilon}(t,x),\nabla_{x}\psi\rangle=\frac{1}{2}\Big\langle\sum_{k=1}^{\infty}\|\sigma_{k}(x)\|^{2} \rho(x)\gamma(x)^{-2}I_{d\times d},\nabla_{x}\psi\Big\rangle+\langle C_{4}(x,\tau,t,\epsilon,\rho),\nabla_{x}\psi\rangle,
\end{eqnarray}
where $C_{4}(x,\tau,t,\epsilon,\rho)$ satisfies
\begin{eqnarray}
|\langle C_{4}(x,\tau,t,\epsilon,\rho),\nabla_{x}\psi\rangle|&\leq& \epsilon|\langle \gamma(x)^{-1}\rho(x)C(x,\tau),\nabla_{x}\psi\rangle|\nonumber\\
&&+\Big|\Big\langle \gamma(x)^{-1}e^{-\frac{\gamma(x)}{\epsilon}(t-t_{k})}\rho(x)\frac{1}{2}\sum_{k=1}^{\infty}\|\sigma_{k}(x)\|^{2}\gamma(x)^{-1},\nabla_{x}\psi\Big\rangle\Big|\nonumber\\
&&+\epsilon|\langle \gamma(x)^{-1}\rho(x)e^{-\frac{\gamma(x)}{\epsilon}(t-t_{k})}C(x,\tau),\nabla_{x}\psi\rangle|\nonumber
\end{eqnarray}
\begin{eqnarray}\label{C4}
&\leq&\epsilon M|\langle \gamma(x)^{-1}\rho(x)(1+\|x\|^{2}),\nabla_{x}\psi\rangle|\nonumber\\
&&+\frac{1}{2}\sup_{x}\sum_{l=1}^{\infty}\|\sigma_{l}(x)\|^{2}e^{-\frac{\gamma_{0}}{\epsilon}(t-t_{k})}|\langle \rho(x)\gamma(x)^{-2},\nabla_{x}\psi\rangle|\nonumber\\
&&+M\epsilon e^{-\frac{\gamma_{0}}{\epsilon}(t-t_{k})}|\langle \gamma(x)^{-1}\rho(x)(1+\|x\|^{2}),\nabla_{x}\psi\rangle|\nonumber\\
&\leq&\epsilon C_{M,\gamma_{0}}|\langle (1+\|x\|^{2})\rho(x),\nabla_{x}\psi\rangle|+\epsilon C_{\gamma_{0},s_{0},L_{\sigma}}|\langle \rho(x),\nabla_{x}\psi\rangle|.
\end{eqnarray}
Furthermore, let 
\begin{eqnarray}\label{YP5}
Z_{3}^{\epsilon}(t,x)&=&\frac{1}{\epsilon}\sum_{l=1}^{\infty}\int_{t_{k}}^{t}\langle e^{-\frac{\gamma(x)}{\epsilon}(t-s)}\sigma_{l}(x)\rho(x),\psi\rangle dB_{s}^{l}\nonumber\\
&=&\sum_{l=1}^{\infty}\int_{\mathbb{R}^{d}}\Big(\frac{1}{\epsilon}\int_{t_{k}}^{t}e^{-\frac{\gamma(x)}{\epsilon}(t-s)}dB_{s}^{l}\Big)\sigma_{l}(x)\rho(x)\psi(x)dx.
\end{eqnarray}
Considering the stationary process 
\begin{eqnarray*}
d\eta_{t}=-\gamma(x)\eta_{t}dt+dB_{t}^{l},
\end{eqnarray*}
 then by a time scale change $t\to\frac{t}{\epsilon}$, we have 
 \begin{eqnarray}\label{LSTJ}
 d\eta^{\epsilon}_{t}=-\frac{\gamma(x)}{\epsilon}\eta^{\epsilon}_{t}dt+\frac{1}{\sqrt{\epsilon}}dB_{t}^{l}.
 \end{eqnarray}
 Define $\dot{u}^{\epsilon}_{t}=\frac{1}{\sqrt{\epsilon}}\eta^{\epsilon}_{t}$. Then direct calculation yields $u^{\epsilon}_{t}-u^{\epsilon}_{t_{k}}=\frac{1}{\sqrt{\epsilon}}\int_{t_{k}}^{t}\eta^{\epsilon}_{s}ds$ 
and 
\begin{eqnarray*}
u^{\epsilon}_{t}-u^{\epsilon}_{t_{k}}-\frac{1}{\gamma(x)}(B^{l}_{t}-B^{l}_{t_{k}})=\sqrt{\epsilon}\Big(\frac{-1}{\gamma(x)}\eta^{\epsilon}_{t}+\frac{1}{\gamma(x)}\eta^{\epsilon}_{t_{k}}\Big).
\end{eqnarray*}
Furthermore, for (\ref{LSTJ}), by a energy method and $(\mathbf{H_{1}})$, we obtain, for every $0\leq t\leq T$, $\mathbb{E}|\eta^{\epsilon}_{t}|^{2}\leq e^{-\frac{2\gamma_{0}}{\epsilon}t}\mathbb{E}|\eta^{\epsilon}_{0}|^{2}$. Thus, 
\begin{eqnarray*}
\frac{1}{\epsilon}\int_{t_{k}}^{t}\int_{t_{k}}^{s}e^{-\frac{\gamma(x)}{\epsilon}(s-u)}dB_{u}^{l}ds-\frac{1}{\gamma(x)}(B^{l}_{t}-B^{l}_{t_{k}})&=&\frac{1}{\sqrt{\epsilon}}\int_{t_{k}}^{t}e^{-\frac{\gamma(x)}{\epsilon}(s-t_{k})}ds\cdot\eta^{\epsilon}_{t_{k}}\\
&&+\sqrt{\epsilon}\Big(\frac{-1}{\gamma(x)}\eta^{\epsilon}_{t}+\frac{1}{\gamma(x)}\eta^{\epsilon}_{t_{k}}\Big). 
\end{eqnarray*}
Note that $\lim_{\epsilon\to0}\frac{1}{\sqrt{\epsilon}}\int_{t_{k}}^{t}e^{-\frac{\gamma(x)}{\epsilon}(s-t_{k})}ds\cdot\eta^{\epsilon}_{t_{k}}=0$ in $L^{2}(\Omega)$, then 
\begin{eqnarray*}\label{LSTJ1}
&&\mathbb{E}\Big|\frac{1}{\epsilon}\int_{t_{k}}^{t}\int_{t_{k}}^{s}e^{-\frac{\gamma(x)}{\epsilon}(s-u)}dB_{u}^{l}ds-\frac{1}{\gamma(x)}(B^{l}_{t}-B^{l}_{t_{k}})\Big|^{2}\nonumber\\
&&\leq \frac{\epsilon}{\gamma_{0}^{2}}e^{-\frac{2\gamma_{0}}{\epsilon}t}\mathbb{E}|\eta^{\epsilon}_{0}|^{2}+\frac{\epsilon}{\gamma_{0}^{2}}e^{-\frac{2\gamma_{0}}{\epsilon}t_{k}}\mathbb{E}|\eta^{\epsilon}_{0}|^{2}\leq C_{\gamma_{0},\eta_{0}}\epsilon e^{-\frac{2\gamma_{0}}{\epsilon}s_{0}}\to0,\quad \epsilon\to0,
\end{eqnarray*}
that is, as $\epsilon\to0$, in the mean square sense,
\begin{eqnarray}\label{LSTJ1}
\frac{1}{\epsilon}\int_{t_{k}}^{t}\int_{t_{k}}^{s}e^{-\frac{\gamma(x)}{\epsilon}(s-u)}dB_{u}^{l}ds=\frac{1}{\gamma(x)}(B^{l}_{t}-B^{l}_{t_{k}})+o(\epsilon e^{-\frac{2\gamma_{0}}{\epsilon}s_{0}}),
\end{eqnarray}
where $o(\epsilon e^{-\frac{2\gamma_{0}}{\epsilon}s_{0}})$ is an infinitely small quantity of $\epsilon$.

Then combining (\ref{YPZ}) with (\ref{YP1}), (\ref{YP2}), (\ref{YP3-1}), (\ref{YP4}) and (\ref{LSTJ1}),
\begin{eqnarray}
&&\int_{t_{k}}^{t}\langle\tilde{Y}_{s}^{\epsilon},\psi\rangle ds\nonumber\\
&=&\epsilon^{2}\int_{t_{k}}^{t}\langle C_{1}(x,s),\psi\rangle ds+\int_{t_{k}}^{t}\langle\gamma(x)^{-1}(K*\rho(x))\rho(x),\psi\rangle ds+\int_{t_{k}}^{t}\langle C_{2}(x,s,\epsilon,\rho),\psi\rangle ds\nonumber\\
&&+\int_{t_{k}}^{t}\Big\langle \frac{1}{2}\sum_{k=1}^{\infty}\|\sigma_{k}(x)\|^{2}\nabla_{x}\gamma(x)\gamma(x)^{-3}\rho(x),\psi\Big\rangle ds+\int_{t_{k}}^{t}\langle C_{3}(x,\tau,s,\epsilon,\rho),\psi\rangle ds\nonumber\\
&&-\int_{t_{k}}^{t}\Big\langle\frac{1}{2}\sum_{k=1}^{\infty}\|\sigma_{k}(x)\|^{2} \rho(x)\gamma(x)^{-2}I_{d\times d},\nabla_{x}\psi\Big\rangle ds-\int_{t_{k}}^{t}\langle C_{4}(x,\tau,s,\epsilon,\rho),\nabla_{x}\psi\rangle ds\nonumber\\
&&+\sum_{l=1}^{\infty}\langle\sigma_{l}(x)\rho(x)\gamma(x)^{-1},\psi\rangle\cdot(B_{t}^{l}-B_{t_{k}}^{l})+\sum_{l=1}^{\infty}\langle\sigma_{l}(x)\rho(x),\psi\rangle\cdot o(\epsilon e^{-\frac{2\gamma_{0}}{\epsilon}s_{0}}).
\end{eqnarray}
Let
\begin{eqnarray}
Y^{*}_{t}(x)&=&\gamma(x)^{-1}(K*\rho(x))\rho(x)+\frac{1}{2}\sum_{l=1}^{\infty}\|\sigma_{l}(x)\|^{2}\nabla_{x}\gamma(x)\gamma(x)^{-3}\rho(x)\nonumber\\
&&+\nabla_{x}\cdot\Big(\frac{1}{2}\sum_{l=1}^{\infty}\|\sigma_{l}(x)\|^{2}\rho(x)\gamma(x)^{-2}I_{d\times d}\Big)+\sum_{l=1}^{\infty}\sigma_{l}(x)\rho(x)\gamma(x)^{-1}\dot{B}_{t}^{l},
\end{eqnarray}
then 
\begin{eqnarray}\label{YTY*}
\mathbb{E}\Big|\int_{t_{k}}^{t}\langle \tilde{Y}_{s}^{\epsilon}-Y^{*}_{s},\psi\rangle ds\Big|&\leq& \epsilon^{2}\int_{t_{k}}^{t}|\langle C_{1}(x,s),\psi\rangle|ds+\int_{t_{k}}^{t}|\langle C_{2}(x,s,\epsilon,\rho),\psi\rangle|ds\nonumber\\
&+&\int_{t_{k}}^{t}|\langle C_{3}(x,\tau,s,\epsilon,\rho),\psi\rangle|ds+\int_{t_{k}}^{t}|\langle C_{4}(x,\tau,s,\epsilon,\rho),\psi\rangle|ds\nonumber\\
&+&\sum_{l=1}^{\infty}|\langle\sigma_{l}(x)\rho(x),\psi\rangle|\cdot o(\sqrt{\epsilon} e^{-\frac{\gamma_{0}}{\epsilon}s_{0}}).
\end{eqnarray}
Thus, by (\ref{YTY*}), (\ref{C1}), (\ref{C2}), (\ref{C3}), (\ref{C4}), together with (\ref{equ:K1}), $(\mathbf{H_{2}})$, Lemma \ref{YS} and the Gaussiality of $\rho$, we have
\begin{eqnarray*}
\mathbb{E}\Big|\int_{t_{k}}^{t}\langle\tilde{Y}_{s}^{\epsilon}-Y^{*}_{s},\psi\rangle ds\Big|&\leq&\sqrt{\epsilon}\delta C_{s_{0},\gamma_{0},M,L_{K},L_{\sigma},X_{0},V_{0}}\|\psi\|_{Lip}+\epsilon\delta C_{s_{0},\gamma_{0}}|\langle (K*\rho(x))\rho(x),\psi\rangle|\nonumber\\
&&+(2\epsilon\delta+\delta^{2})C_{M,\gamma_{0},L_{\gamma}}|\langle(1+\|x\|^{2})\rho(x),\psi\rangle|\nonumber\\
&&+\Big(\epsilon\delta+\frac{\delta^{2}}{\epsilon}\Big)C_{\gamma_{0},s_{0},L_{\sigma},L_{\gamma}}|\langle\rho(x),\psi\rangle|\nonumber\\
&&+\epsilon\delta C_{M,\gamma_{0}}|\langle (1+\|x\|^{2})\rho(x),\nabla_{x}\psi\rangle|+\epsilon\delta C_{\gamma_{0},s_{0},L_{\sigma}}|\langle \rho(x),\nabla_{x}\psi\rangle|\nonumber\\
&&+C_{L_{\sigma}}|\langle\rho,|\psi|\rangle|o(1,\sqrt{\epsilon} e^{-\frac{\gamma_{0}}{\epsilon}s_{0}}).
\end{eqnarray*}
\end{proof}
Now we are in the position to prove our main result Theorem \ref{main2}. 
Without loss of generality, by the tightness of $\{X_{t}^{\epsilon,\mu^{\epsilon}}\}$, we assume $\{\rho_{\cdot}^{\epsilon}\}$ converges weakly to $\{\rho_{\cdot}\}$ in space $C(0,T;\mathbb{R}^{d})$ as $\epsilon\to0$. (If necessary one can choose a subsequence of $\{\rho_{\cdot}^{\epsilon}\}$). Now we determine the equation satisfied by $\rho_{\cdot}$. Firstly we rewrite $Y^{*}$ as $Y^{*,\rho}$ to emphasize the dependence on $\rho$ and $\tilde{Y}_{t}^{\epsilon,\rho,\tau}$ as the solution to (\ref{FU2}). Then on the interval $[t_{k},t_{k+1}]$, $\{\hat{Y}_{t}^{\epsilon}\}=\{\tilde{Y}_{t}^{\epsilon,\rho_{t_{k}}^{\epsilon},t_{k}}\}$. Now, let $\psi=\nabla_{x}\phi,\phi\in C_{0}^{\infty}(\mathbb{R}^{d})$, then
\begin{eqnarray}
&&\langle \rho_{t}^{\epsilon},\phi\rangle=\langle  \rho_{0},\phi\rangle+\int_{0}^{t}\langle Y_{s}^{\epsilon},\psi\rangle ds=\langle  \rho_{0},\phi\rangle+\int_{0}^{t}\langle Y^{*, \rho_{s}^{\epsilon}},\psi\rangle ds\nonumber\\
&&+\int_{0}^{t}\langle Y_{s}^{\epsilon}-\hat{Y}_{s}^{\epsilon},\psi\rangle ds+\int_{0}^{t}\langle \hat{Y}_{s}^{\epsilon}-Y^{*, \rho_{s}^{\epsilon}},\psi\rangle ds\nonumber\\
&&=\langle  \rho_{0},\phi\rangle+\int_{0}^{t}\langle Y^{*,\rho_{s}^{\epsilon}},\psi\rangle ds+\int_{0}^{t}\langle Y_{s}^{\epsilon}-\hat{Y}_{s}^{\epsilon},\psi\rangle ds\nonumber\\
&&+\sum_{k=0}^{\lfloor \frac{t}{\delta}\rfloor-1}\int_{t_{k}}^{t_{k+1}}\langle \tilde{Y}_{s}^{\epsilon,\rho_{t_{k}}^{\epsilon},t_{k}}-Y^{*,\rho_{t_{k}}^{\epsilon}}, \psi\rangle ds+\int_{t_{\lfloor \frac{t}{\delta}\rfloor}}^{t}\langle\tilde{Y}_{s}^{\epsilon,\rho^{\epsilon}_{t_{\lfloor \frac{t}{\delta}\rfloor}},t_{\lfloor \frac{t}{\delta}\rfloor}}-Y^{*,\rho^{\epsilon}_{t_{\lfloor \frac{t}{\delta}\rfloor}}},\psi\rangle ds\nonumber\\
&&+\sum_{k=0}^{\lfloor \frac{t}{\delta}\rfloor-1}\int_{t_{k}}^{t_{k+1}}\langle Y^{*, \rho_{t_{k}}^{\epsilon}}-Y^{*, \rho_{s}^{\epsilon}},\psi\rangle ds+\int_{t_{\lfloor \frac{t}{\delta}\rfloor}}^{t}\langle Y^{*, \rho_{t_{\lfloor \frac{t}{\delta}\rfloor}}^{\epsilon}}-Y^{*, \rho_{s}^{\epsilon}},\psi\rangle ds.
\end{eqnarray}
Firstly, note that
\begin{eqnarray}\label{F4} 
 &&\int_{0}^{t}\langle Y^{*, \rho_{s}^{\epsilon}}_{s},\psi\rangle ds=\sum_{k=0}^{\lfloor \frac{t}{\delta}\rfloor-1}\int_{t_{k}}^{t_{k+1}}\langle Y^{*,\rho_{t_{k}}^{\epsilon}}_{s},\psi\rangle ds+\int_{t_{\lfloor \frac{t}{\delta}\rfloor}}^{t}\langle Y^{*,\rho^{\epsilon}_{t_{\lfloor \frac{t}{\delta}\rfloor}}}_{s},\psi\rangle ds\nonumber\\
&=&\int_{0}^{t}\Big[\langle\gamma(x)^{-1}(K*\rho_{s}^{\epsilon}(x))\rho_{s}^{\epsilon}(x),\psi\rangle+\langle\frac{1}{2}\sum_{l=1}^{\infty}\|\sigma_{l}(x)\|^{2}\nabla_{x}\gamma(x)\gamma(x)^{-3}\rho_{s}^{\epsilon}(x),\psi\rangle\nonumber\\
&&-\langle\frac{1}{2}\sum_{l=1}^{\infty}\|\sigma_{l}(x)\|^{2}\rho_{s}^{\epsilon}(x)\gamma(x)^{-2}I_{d\times d},\nabla_{x}\psi\rangle\Big]ds+\langle\sum_{l=1}^{\infty}\sigma_{l}(x)\rho^{\epsilon}_{t_{\lfloor \frac{t}{\delta}\rfloor}}\gamma(x)^{-1}(B_{t}^{l}-B_{t_{\lfloor \frac{t}{\delta}\rfloor}}^{l}),\psi\rangle\nonumber\\
&&+\langle\sum_{l=1}^{\infty}\sigma_{l}(x)\gamma(x)^{-1}\sum_{k=0}^{\lfloor \frac{t}{\delta}\rfloor-1}\rho_{t_{k}}^{\epsilon}(x)(B_{t_{k+1}}^{l}-B_{t_{k}}^{l}),\psi\rangle.
\end{eqnarray} 
By the definition of $Y^{*,\rho}$ in Lemma \ref{YR}, and the tightness of $\{X_{t}^{\epsilon,\mu^{\epsilon}}\}$ in space $C(0,T;\mathbb{R}^{d})$, for any sequence of $\{X_{t}^{\epsilon,\mu^{\epsilon}}\}$, there exists a subsequence denoted still by $\{X_{t}^{\epsilon,\mu^{\epsilon}}\}$ such that $\{X_{t}^{\epsilon,\mu^{\epsilon}}\}$ converges in distribution to $X(t)$, $\epsilon\to0$. By Skorohod theorem, we still can construct a new probability space and random variables without changing the distribution such that (here we don't changing the notations) $\{X_{t}^{\epsilon,\mu^{\epsilon}}\}$ converges almost surely to $X(t)$. Then by assumptions $(\mathbf{H_{1}})$-$(\mathbf{H_{4}})$ and the conditional dominate convergence theorem, we obtain
\begin{eqnarray}\label{F4-1} 
&&\int_{0}^{t}\Big[\langle\gamma(x)^{-1}(K*\rho_{s}^{\epsilon}(x))\rho_{s}^{\epsilon}(x),\psi\rangle+\langle\frac{1}{2}\sum_{l=1}^{\infty}\|\sigma_{l}(x)\|^{2}\nabla_{x}\gamma(x)\gamma(x)^{-3}\rho_{s}^{\epsilon}(x),\psi\rangle\\
&&\quad\quad-\langle\frac{1}{2}\sum_{l=1}^{\infty}\|\sigma_{l}(x)\|^{2}\rho_{s}^{\epsilon}(x)\gamma(x)^{-2}I_{d\times d},\nabla_{x}\psi\rangle\Big]ds\to\int_{0}^{t}\Big[\langle\gamma(x)^{-1}(K*\rho_{s}(x))\rho_{s}(x),\psi\rangle\nonumber\\
&&\quad\quad+\langle\frac{1}{2}\sum_{l=1}^{\infty}\|\sigma_{l}(x)\|^{2}\nabla_{x}\gamma(x)\gamma(x)^{-3}\rho_{s}(x),\psi\rangle-\langle\frac{1}{2}\sum_{l=1}^{\infty}\|\sigma_{l}(x)\|^{2}\rho_{s}(x)\gamma(x)^{-2}I_{d\times d},\nabla_{x}\psi\rangle\Big]ds.\nonumber
\end{eqnarray} 
Further, by the integrability of $\rho_{\cdot}^{\epsilon}$ and the definition of It${\rm\hat{o}}$'s integration, in the mean square sense, 
\begin{eqnarray*}
\sum_{k=0}^{\lfloor \frac{t}{\delta}\rfloor-1}\rho_{t_{k}}^{\epsilon}(x)(B_{t_{k+1}}^{l}-B_{t_{k}}^{l})\to\int_{0}^{t}\rho_{s}^{\epsilon}(x)dB_{s}^{l}, \quad\delta\to0,
\end{eqnarray*}
that is, as $\delta\to0$, in the mean square sense,
\begin{eqnarray}\label{F4-2} 
\sum_{k=0}^{\lfloor \frac{t}{\delta}\rfloor-1}\rho_{t_{k}}^{\epsilon}(x)(B_{t_{k+1}}^{l}-B_{t_{k}}^{l})=\int_{0}^{t}\rho_{s}^{\epsilon}(x)dB_{s}^{l}+o(\delta),
\end{eqnarray} 
where $o(\delta)$ is an infinitely small quantity of $\delta$.
Then
\begin{eqnarray}\label{F4-3} 
&&\langle\sum_{l=1}^{\infty}\sigma_{l}(x)\gamma(x)^{-1}\sum_{k=0}^{\lfloor \frac{t}{\delta}\rfloor-1}\rho_{t_{k}}^{\epsilon}(x)(B_{t_{k+1}}^{l}-B_{t_{k}}^{l}),\psi\rangle\\
&=&\sum_{l=1}^{\infty}\int_{0}^{t}\langle\sigma_{l}(x)\gamma(x)^{-1}\rho_{s}^{\epsilon}(x),\psi\rangle dB_{s}^{l}+\sum_{l=1}^{\infty}\langle\sigma_{l}(x)\gamma(x)^{-1},\psi\rangle\cdot o(\delta).\nonumber
\end{eqnarray} 
By $(\mathbf{H_{2}})$ and $(\mathbf{H_{3}})$,
\begin{eqnarray}\label{F4-4} 
\sum_{l=1}^{\infty}|\langle\sigma_{l}(x)\gamma(x)^{-1},\psi\rangle|\cdot o(\delta)\leq C_{\gamma_{0},L_{\sigma}} \|\psi\|_{Lip}\cdot o(\delta).
\end{eqnarray}
Furthermore, by the H${\rm\ddot{o}}$lder inequality, 
\begin{eqnarray}
&&\mathbb{E}\sum_{l=1}^{\infty}|\langle\sigma_{l}(x)\rho^{\epsilon}_{t_{\lfloor \frac{t}{\delta}\rfloor}}\gamma(x)^{-1}(B_{t}^{l}-B_{t_{\lfloor \frac{t}{\delta}\rfloor}}^{l}),\psi\rangle|\nonumber\\
&\leq& \sum_{l=1}^{\infty}(\mathbb{E}|\langle\sigma_{l}(x)\rho^{\epsilon}_{t_{\lfloor \frac{t}{\delta}\rfloor}}\gamma(x)^{-1},\psi\rangle|^{2})^{\frac{1}{2}}(\mathbb{E}|B_{t}^{l}-B_{t_{\lfloor \frac{t}{\delta}\rfloor}}^{l}|^{2})^{\frac{1}{2}}\nonumber
\end{eqnarray}
\begin{eqnarray}\label{F4-5}
&\leq&\sum_{l=1}^{\infty}(\mathbb{E}|\mathbb{E}[\sigma_{l}(X_{t_{\lfloor \frac{t}{\delta}\rfloor}}^{\epsilon,\mu^{\epsilon}})\gamma(X_{t_{\lfloor \frac{t}{\delta}\rfloor}}^{\epsilon,\mu^{\epsilon}})^{-1}\psi(X_{t_{\lfloor \frac{t}{\delta}\rfloor}}^{\epsilon,\mu^{\epsilon}})|\mathcal{F}_{t_{\lfloor \frac{t}{\delta}\rfloor}}^{B}]|^{2})^{\frac{1}{2}}(|t-t_{\lfloor \frac{t}{\delta}\rfloor}|)^{\frac{1}{2}}\nonumber\\
&\leq&\sum_{l=1}^{\infty}(\mathbb{E}|\sigma_{l}(X_{t_{\lfloor \frac{t}{\delta}\rfloor}}^{\epsilon,\mu^{\epsilon}})\gamma(X_{t_{\lfloor \frac{t}{\delta}\rfloor}}^{\epsilon,\mu^{\epsilon}})^{-1}\psi(X_{t_{\lfloor \frac{t}{\delta}\rfloor}}^{\epsilon,\mu^{\epsilon}})|^{2})^{\frac{1}{2}}\delta^{\frac{1}{2}}\nonumber\\
&\leq&\frac{1}{\gamma_{0}}\|\psi\|_{Lip}(\sum_{l=1}^{\infty}\sup_{x\in\mathbb{R}^{d}}\|\sigma_{l}(x)\|^{2})^{\frac{1}{2}}\delta^{\frac{1}{2}}\leq C_{\gamma_{0},L_{\sigma}}\|\psi\|_{Lip}\delta^{\frac{1}{2}}.
\end{eqnarray}
By (\ref{F4})-(\ref{F4-5}), we have
\begin{eqnarray}\label{F4-6}
\int_{0}^{t}\langle Y^{*, \rho_{s}^{\epsilon}}_{s},\psi\rangle ds\to\int_{0}^{t}\langle Y^{*, \rho_{s}}_{s},\psi\rangle ds,\quad \epsilon\to0, \delta\to0,
\end{eqnarray}
where
\begin{eqnarray*}
Y^{*, \rho_{s}}_{s}&=&\gamma(x)^{-1}(K*\rho_{s}(x))\rho_{s}(x)+\frac{1}{2}\sum_{l=1}^{\infty}\|\sigma_{l}(x)\|^{2}\nabla_{x}\gamma(x)\gamma(x)^{-3}\rho_{s}(x)\\
&&+\nabla_{x}\cdot\Big(\frac{1}{2}\sum_{l=1}^{\infty}\|\sigma_{l}(x)\|^{2}\rho_{s}(x)\gamma(x)^{-2}\cdot I_{d\times d}\Big)+\sum_{l=1}^{\infty}\sigma_{l}(x)\rho_{s}(x)\gamma(x)^{-1}\dot{B}_{s}^{l}.
\end{eqnarray*}
Further, by Lemma \ref{LC1},
\begin{eqnarray}\label{LL1}
\mathbb{E}\Big|\int_{0}^{t}\langle Y_{s}^{\epsilon}-\hat{Y}_{s}^{\epsilon},\psi\rangle ds\Big|&\leq& C_{T}\Big(\frac{\sqrt{\delta}}{\epsilon}+\frac{\delta}{\epsilon}+\frac{\delta}{\epsilon^{2}}\Big)\|\psi\|_{Lip}.
\end{eqnarray}
Note that by Lemma \ref{YR},
\begin{eqnarray}\label{FDH}
&&\mathbb{E}\Big|\sum_{k=0}^{\lfloor \frac{t}{\delta}\rfloor-1}\int_{t_{k}}^{t_{k+1}}\langle \tilde{Y}_{s}^{\epsilon,\rho_{t_{k}}^{\epsilon},t_{k}}-Y^{*,\rho_{t_{k}}^{\epsilon}}_{s}, \psi\rangle ds+\int_{t_{\lfloor \frac{t}{\delta}\rfloor}}^{t}\langle\tilde{Y}_{s}^{\epsilon,\rho^{\epsilon}_{t_{\lfloor \frac{t}{\delta}\rfloor}},t_{\lfloor \frac{t}{\delta}\rfloor}}-Y^{*,\rho^{\epsilon}_{t_{\lfloor \frac{t}{\delta}\rfloor}}},\psi\rangle ds\Big|\nonumber\\
&\leq&\sum_{k=0}^{\lfloor \frac{t}{\delta}\rfloor-1}\mathbb{E}\Big|\int_{t_{k}}^{t_{k+1}}\langle\tilde{Y}_{s}^{\epsilon,\rho_{t_{k}}^{\epsilon},t_{k}}-Y^{*,\rho_{t_{k}}^{\epsilon}}, \psi\rangle ds\Big|+\mathbb{E}\Big|\int_{t_{\lfloor \frac{t}{\delta}\rfloor}}^{t}\langle\tilde{Y}_{s}^{\epsilon,\rho^{\epsilon}_{t_{\lfloor \frac{t}{\delta}\rfloor}},t_{\lfloor \frac{t}{\delta}\rfloor}}-Y^{*,\rho^{\epsilon}_{t_{\lfloor \frac{t}{\delta}\rfloor}}},\psi\rangle ds\Big|\nonumber\\
&\leq&\sum_{k=0}^{\lfloor \frac{t}{\delta}\rfloor}\Big[\sqrt{\epsilon}\delta C_{s_{0},\gamma_{0},M,L_{K},L_{\sigma},X_{0},V_{0}}\|\psi\|_{Lip}+\epsilon\delta C_{s_{0},\gamma_{0}}|\langle (K*\rho_{t_{k}}^{\epsilon})\rho_{t_{k}}^{\epsilon}(x),\psi\rangle|\nonumber\\
&&+(2\epsilon\delta+\delta^{2})C_{M,\gamma_{0},L_{\gamma}}|\langle(1+\|x\|^{2})\rho_{t_{k}}^{\epsilon}(x),\psi\rangle|+\Big(\epsilon\delta+\frac{\delta^{2}}{\epsilon}\Big)C_{\gamma_{0},s_{0},L_{\sigma},L_{\gamma}}|\langle\rho_{t_{k}}^{\epsilon}(x),\psi\rangle|\nonumber\\
&&+\epsilon\delta C_{M,\gamma_{0}}|\langle (1+\|x\|^{2})\rho_{t_{k}}^{\epsilon}(x),\nabla_{x}\psi\rangle|+\epsilon\delta C_{\gamma_{0},s_{0},L_{\sigma}}|\langle \rho_{t_{k}}^{\epsilon}(x),\nabla_{x}\psi\rangle|\nonumber\\
&&+C_{L_{\sigma}}|\langle\rho_{t_{k}}^{\epsilon},|\psi|\rangle|o(\sqrt{\epsilon} e^{-\frac{\gamma_{0}}{\epsilon}s_{0}})\Big]\nonumber\\
&\leq&C_{\gamma_{0},s_{0},L_{\sigma},L_{\gamma},L_{K},M,\|\psi\|_{\infty},X_{0},V_{0}}\|\psi\|_{Lip}\Big(7\epsilon+\delta+\frac{\delta}{\epsilon}\Big)+C_{L_{\sigma}}\Big(1+\frac{T}{\delta}\Big)o(\sqrt{\epsilon} e^{-\frac{\gamma_{0}}{\epsilon}s_{0}}).
\end{eqnarray}
Similarly, 
\begin{eqnarray}
&&\mathbb{E}\Big|\sum_{k=0}^{\lfloor \frac{t}{\delta}\rfloor-1}\int_{t_{k}}^{t_{k+1}}\langle Y^{*,\rho_{t_{k}}^{\epsilon}}-Y^{*,\rho_{s}^{\epsilon}}, \psi\rangle ds+\int_{t_{\lfloor \frac{t}{\delta}\rfloor}}^{t}\langle Y^{*,\rho^{\epsilon}_{t_{\lfloor \frac{t}{\delta}\rfloor}}}-Y^{*,\rho_{s}^{\epsilon}},\psi\rangle ds\Big|\nonumber\\
&\leq&\sum_{k=0}^{\lfloor \frac{t}{\delta}\rfloor-1}\mathbb{E}\int_{t_{k}}^{t_{k+1}}\Big[\Big|\Big\langle\gamma(x)^{-1}(K*\rho_{t_{k}}^{\epsilon})\rho_{t_{k}}^{\epsilon}+\frac{1}{2}\sum_{k=1}^{\infty}\|\sigma_{k}(x)\|^{2}\nabla_{x}\gamma(x)\gamma(x)^{-3}\rho_{t_{k}}^{\epsilon}\nonumber\\
&&+\frac{1}{2}\nabla_{x}\cdot\Big(\sum_{k=1}^{\infty}\|\sigma_{k}(x)\|^{2}\rho_{t_{k}}^{\epsilon}\gamma(x)^{-2}\cdot I_{d\times d}\Big)+\sum_{l=1}^{\infty}\sigma_{l}(x)\rho_{t_{k}}^{\epsilon}\gamma(x)^{-1}\dot{B}_{s}^{l},\psi\Big\rangle\nonumber
\end{eqnarray}
\begin{eqnarray}\label{FDH1}
&&-\Big\langle \gamma(x)^{-1}(K*\rho_{s}^{\epsilon})\rho_{s}^{\epsilon}(x)+\frac{1}{2}\sum_{k=1}^{\infty}\|\sigma_{k}(x)\|^{2}\nabla_{x}\gamma(x)\gamma(x)^{-3}\rho_{s}^{\epsilon}\nonumber\\
&&+\frac{1}{2}\nabla_{x}\cdot\Big(\sum_{k=1}^{\infty}\|\sigma_{k}(x)\|^{2}\gamma(x)^{-2}\rho_{s}^{\epsilon}\cdot I_{d\times d}\Big)+\sum_{l=1}^{\infty}\sigma_{l}(x)\rho_{s}^{\epsilon}\gamma(x)^{-1}\dot{B}_{s}^{l},\psi\Big\rangle\Big|\Big]ds\nonumber\\
&&+\mathbb{E}\int_{t_{\lfloor \frac{t}{\delta}\rfloor}}^{t}\Big[\Big|\Big\langle\gamma(x)^{-1}(K*\rho_{t_{\lfloor \frac{t}{\delta}\rfloor}}^{\epsilon})\rho_{t_{\lfloor \frac{t}{\delta}\rfloor}}^{\epsilon}+\frac{1}{2}\sum_{k=1}^{\infty}\|\sigma_{k}(x)\|^{2}\nabla_{x}\gamma(x)\gamma(x)^{-3}\rho_{t_{\lfloor \frac{t}{\delta}\rfloor}}^{\epsilon}\nonumber\\
&&+\frac{1}{2}\nabla_{x}\cdot\Big(\sum_{k=1}^{\infty}\|\sigma_{k}(x)\|^{2}\rho_{t_{\lfloor \frac{t}{\delta}\rfloor}}^{\epsilon}\gamma(x)^{-2}\cdot I_{d\times d}\Big)+\sum_{l=1}^{\infty}\sigma_{l}(x)\rho_{t_{\lfloor \frac{t}{\delta}\rfloor}}^{\epsilon}\gamma(x)^{-1}\dot{B}_{s}^{l},\psi\Big\rangle\nonumber\\
&&-\Big\langle \gamma(x)^{-1}(K*\rho_{s}^{\epsilon})\rho_{s}^{\epsilon}(x)+\frac{1}{2}\sum_{k=1}^{\infty}\|\sigma_{k}(x)\|^{2}\nabla_{x}\gamma(x)\gamma(x)^{-3}\rho_{s}^{\epsilon}\nonumber\\
&&+\frac{1}{2}\nabla_{x}\cdot\Big(\sum_{k=1}^{\infty}\|\sigma_{k}(x)\|^{2}\gamma(x)^{-2}\rho_{s}^{\epsilon}\cdot I_{d\times d}\Big)+\sum_{l=1}^{\infty}\sigma_{l}(x)\rho_{s}^{\epsilon}\gamma(x)^{-1}\dot{B}_{s}^{l},\psi\Big\rangle\Big|\Big]ds\nonumber\\
&\leq&\sum_{k=0}^{\lfloor \frac{t}{\delta}\rfloor}\mathbb{E}\int_{t_{k}}^{t_{k+1}}|\langle \gamma(x)^{-1}(K*\rho_{t_{k}}^{\epsilon})\rho_{t_{k}}^{\epsilon}(x)-\gamma(x)^{-1}(K*\rho_{s}^{\epsilon})\rho_{s}^{\epsilon}(x),\psi\rangle|ds\nonumber\\
&&+\sum_{k=0}^{\lfloor \frac{t}{\delta}\rfloor}\mathbb{E}\int_{t_{k}}^{t_{k+1}}\Big|\Big\langle \frac{1}{2}\sum_{l=1}^{\infty}\|\sigma_{l}(x)\|^{2}\nabla_{x}\gamma(x)\gamma(x)^{-3}(\rho_{t_{k}}^{\epsilon}-\rho_{s}^{\epsilon}),\psi\Big\rangle\Big|ds\nonumber\\
&&+\sum_{k=0}^{\lfloor \frac{t}{\delta}\rfloor}\mathbb{E}\int_{t_{k}}^{t_{k+1}}\Big|\Big\langle \frac{1}{2}\sum_{l=1}^{\infty}\|\sigma_{l}(x)\|^{2}(\rho_{t_{k}}^{\epsilon}-\rho_{s}^{\epsilon})\gamma(x)^{-2}\cdot I_{d\times d},\nabla_{x}\psi\Big\rangle\Big|ds\nonumber\\
&&+\mathbb{E}\sum_{k=0}^{\lfloor \frac{t}{\delta}\rfloor}\Big|\int_{t_{k}}^{t_{k+1}}\Big\langle\sum_{l=1}^{\infty}\sigma_{l}(x)\gamma(x)^{-1}(\rho_{t_{k}}^{\epsilon}-\rho_{s}^{\epsilon}),\psi\Big\rangle dB_{s}^{l}\Big|\nonumber\\
&\triangleq&\sum_{k=0}^{\lfloor \frac{t}{\delta}\rfloor}\mathbb{E}\int_{t_{k}}^{t_{k+1}}\sum_{i=1}^{3}L_{i}^{\epsilon}(s)ds+\mathbb{E}\sum_{k=0}^{\lfloor \frac{t}{\delta}\rfloor}\Big|\int_{t_{k}}^{t_{k+1}}\Big\langle\sum_{l=1}^{\infty}\sigma_{l}(x)\gamma(x)^{-1}(\rho_{t_{k}}^{\epsilon}-\rho_{s}^{\epsilon}),\psi\Big\rangle dB_{s}^{l}\Big|.
\end{eqnarray}
Next we estimate the four terms. 

For $L_{1}^{\epsilon}(s)$, by the definition of expectation and $(\mathbf{H_{2}})$, we have
\begin{eqnarray}
L_{1}^{\epsilon}(s)&=&|\langle \gamma(x)^{-1}(K*\rho_{t_{k}}^{\epsilon})\rho_{t_{k}}^{\epsilon}(x)-\gamma(x)^{-1}(K*\rho_{s}^{\epsilon})\rho_{s}^{\epsilon}(x),\psi\rangle|\nonumber\\
&\leq&|\langle \gamma(x)^{-1}(K*\rho_{t_{k}}^{\epsilon}-K*\rho_{s}^{\epsilon})\rho_{t_{k}}^{\epsilon}(x),\psi\rangle|\nonumber\\
&&+|\langle \gamma(x)^{-1}(K*\rho_{s}^{\epsilon})(\rho_{t_{k}}^{\epsilon}(x)-\rho_{s}^{\epsilon}(x)),\psi\rangle|\nonumber\\
&\leq&|\langle \gamma(x)^{-1}(\mathbb{E}_{\bar{X}_{t_{k}}^{\epsilon,\mu^{\epsilon}}}[K(x-\bar{X}_{t_{k}}^{\epsilon,\mu^{\epsilon}})|\mathcal{F}_{t_{k}}^{B}]-\mathbb{E}_{\bar{X}_{s}^{\epsilon,\mu^{\epsilon}}}[K(x-\bar{X}_{s}^{\epsilon,\mu^{\epsilon}})|\mathcal{F}_{s}^{B}])\rho_{t_{k}}^{\epsilon}(x),\psi\rangle|\nonumber\\
&&+|\langle \gamma(x)^{-1}\mathbb{E}_{\bar{X}_{s}^{\epsilon,\mu^{\epsilon}}}[K(x-\bar{X}_{s}^{\epsilon,\mu^{\epsilon}})|\mathcal{F}_{s}^{B}](\rho_{t_{k}}^{\epsilon}-\rho_{s}^{\epsilon}),\psi\rangle|\nonumber\\
&=&|\mathbb{E}_{X_{t_{k}}^{\epsilon,\mu^{\epsilon}}}[\gamma(X_{t_{k}}^{\epsilon,\mu^{\epsilon}})^{-1}(\mathbb{E}_{\bar{X}_{t_{k}}^{\epsilon,\mu^{\epsilon}}}[K(X_{t_{k}}^{\epsilon,\mu^{\epsilon}}-\bar{X}_{t_{k}}^{\epsilon,\mu^{\epsilon}})|\mathcal{F}_{t_{k}}^{B}]\nonumber
\end{eqnarray}
\begin{eqnarray}\label{L1}
&&-\mathbb{E}_{\bar{X}_{s}^{\epsilon,\mu^{\epsilon}}}[K(X_{t_{k}}^{\epsilon,\mu^{\epsilon}}-\bar{X}_{s}^{\epsilon,\mu^{\epsilon}})|\mathcal{F}_{s}^{B}])\psi(X_{t_{k}}^{\epsilon,\mu^{\epsilon}})|\mathcal{F}_{t_{k}}^{B}]|\nonumber\\
&&+|\mathbb{E}_{X_{t_{k}}^{\epsilon,\mu^{\epsilon}}}[\gamma(X_{t_{k}}^{\epsilon,\mu^{\epsilon}})^{-1}\mathbb{E}_{\bar{X}_{s}^{\epsilon,\mu^{\epsilon}}}[K(X_{t_{k}}^{\epsilon,\mu^{\epsilon}}-\bar{X}_{s}^{\epsilon,\mu^{\epsilon}})|\mathcal{F}_{s}^{B}]\psi(X_{t_{k}}^{\epsilon,\mu^{\epsilon}})|\mathcal{F}_{t_{k}}^{B}]\nonumber\\
&&-\mathbb{E}_{X_{s}^{\epsilon,\mu^{\epsilon}}}[\gamma(X_{s}^{\epsilon,\mu^{\epsilon}})^{-1}\mathbb{E}_{\bar{X}_{s}^{\epsilon,\mu^{\epsilon}}}[K(X_{s}^{\epsilon,\mu^{\epsilon}}-\bar{X}_{s}^{\epsilon,\mu^{\epsilon}})|\mathcal{F}_{s}^{B}]\psi(X_{s}^{\epsilon,\mu^{\epsilon}})|\mathcal{F}_{s}^{B}]|\nonumber\\
&\triangleq&L_{1,1}^{\epsilon}(s)+L_{1,2}^{\epsilon}(s)
\end{eqnarray}
By (\ref{MKD}), for each $\pi\in\Pi(\mathcal{L}(\bar{X}_{t_{k}}^{\epsilon,\mu^{\epsilon}}|\mathcal{F}_{t_{k}}^{B}),\mathcal{L}(\bar{X}_{s}^{\epsilon,\mu^{\epsilon}}|\mathcal{F}_{s}^{B}))=\Pi(\mathcal{L}(\bar{X}_{t_{k}}^{\epsilon,\mu^{\epsilon}}|\mathcal{F}_{s}^{B}),\mathcal{L}(\bar{X}_{s}^{\epsilon,\mu^{\epsilon}}|\mathcal{F}_{s}^{B}))$, together with $(\mathbf{H_{1}})$, 
\begin{eqnarray*}
L_{1,1}^{\epsilon}(s)&\leq&\frac{\|\psi\|_{Lip}}{\gamma_{0}}\mathbb{E}_{X_{t_{k}}^{\epsilon,\mu^{\epsilon}}}[\|\mathbb{E}_{\bar{X}_{t_{k}}^{\epsilon,\mu^{\epsilon}}}[K(X_{t_{k}}^{\epsilon,\mu^{\epsilon}}-\bar{X}_{t_{k}}^{\epsilon,\mu^{\epsilon}})|\mathcal{F}_{t_{k}}^{B}]-\mathbb{E}_{\bar{X}_{s}^{\epsilon,\mu^{\epsilon}}}[K(X_{t_{k}}^{\epsilon,\mu^{\epsilon}}-\bar{X}_{s}^{\epsilon,\mu^{\epsilon}})|\mathcal{F}_{s}^{B}]\||\mathcal{F}_{t_{k}}^{B}]\nonumber\\
&=&\frac{\|\psi\|_{Lip}}{\gamma_{0}}\mathbb{E}_{X_{t_{k}}^{\epsilon,\mu^{\epsilon}}}\Big[\Big\|\int_{\mathbb{R}^{d}}K(X_{t_{k}}^{\epsilon,\mu^{\epsilon}}-y)\mathcal{L}(\bar{X}_{t_{k}}^{\epsilon,\mu^{\epsilon}}|\mathcal{F}_{t_{k}}^{B})(dy)\nonumber\\
&&-\int_{\mathbb{R}^{d}}K(X_{t_{k}}^{\epsilon,\mu^{\epsilon}}-y)\mathcal{L}(\bar{X}_{s}^{\epsilon,\mu^{\epsilon}}|\mathcal{F}_{s}^{B})(dy)\Big\|\Big|\mathcal{F}_{t_{k}}^{B}\Big]\nonumber\\
&=&\frac{\|\psi\|_{Lip}}{\gamma_{0}}\mathbb{E}\Big[\Big\|\int_{\mathbb{R}^{d}\times\mathbb{R}^{d}}K(X_{t_{k}}^{\epsilon,\mu^{\epsilon}}-y)\pi(dy,dz)-\int_{\mathbb{R}^{d}\times\mathbb{R}^{d}}K(X_{t_{k}}^{\epsilon,\mu^{\epsilon}}-z)\pi(dy,dz)\Big\|\Big|\mathcal{F}_{t_{k}}^{B}\Big]\nonumber\\
&\leq&C_{\gamma_{0},L_{K}}\|\psi\|_{Lip}\int_{\mathbb{R}^{d}\times\mathbb{R}^{d}}\|y-z\|\pi(dy,dz),
\end{eqnarray*}
then together with (\ref{MKD}) and (\ref{equ:2.101}),
\begin{eqnarray}\label{L1-1}
\mathbb{E}L_{1,1}^{\epsilon}(s)&\leq&C_{\gamma_{0},\|\psi\|_{\infty},L_{K}}\mathbb{E}\|X_{t_{k}}^{\epsilon,\mu^{\epsilon}}-X_{s}^{\epsilon,\mu^{\epsilon}}\|^{2}\|\psi\|_{Lip}.
\end{eqnarray}
Similarly, for $L_{1,2}^{\epsilon}(s)$, taking any $\pi\in\Pi(\mathcal{L}(X_{t_{k}}^{\epsilon,\mu^{\epsilon}}|\mathcal{F}_{s}^{B}),\mathcal{L}(X_{s}^{\epsilon,\mu^{\epsilon}}|\mathcal{F}_{s}^{B}))$,
\begin{eqnarray*}
L_{1,2}^{\epsilon}(s)&=&\|\mathbb{E}_{X_{t_{k}}^{\epsilon,\mu^{\epsilon}}}[\gamma(X_{t_{k}}^{\epsilon,\mu^{\epsilon}})^{-1} \mathbb{E}_{\bar{X}_{s}^{\epsilon,\mu^{\epsilon}}}[K(X_{t_{k}}^{\epsilon,\mu^{\epsilon}}-\bar{X}_{s}^{\epsilon,\mu^{\epsilon}})|\mathcal{F}_{s}^{B}]\psi(X_{t_{k}}^{\epsilon,\mu^{\epsilon}})|\mathcal{F}_{t_{k}}^{B}]\nonumber\\
&&-\mathbb{E}_{X_{s}^{\epsilon,\mu^{\epsilon}}}[\gamma(X_{s}^{\epsilon,\mu^{\epsilon}})^{-1}\mathbb{E}_{\bar{X}_{s}^{\epsilon,\mu^{\epsilon}}}[K(X_{s}^{\epsilon,\mu^{\epsilon}}-\bar{X}_{s}^{\epsilon,\mu^{\epsilon}})|\mathcal{F}_{s}^{B}]\psi(X_{s}^{\epsilon,\mu^{\epsilon}})|\mathcal{F}_{s}^{B}]\|\nonumber\\
&\leq&\Big\|\int_{\mathbb{R}^{d}\times\mathbb{R}^{d}}(\gamma(x)^{-1}-\gamma(z)^{-1})\mathbb{E}[K(x-\bar{X}_{s}^{\epsilon,\mu^{\epsilon}})|\mathcal{F}_{s}^{B}]\psi(x)\pi(dx,dz)\Big\|\nonumber\\
&&+\Big\|\int_{\mathbb{R}^{d}\times\mathbb{R}^{d}}\gamma(z)^{-1}\mathbb{E}[K(x-\bar{X}_{s}^{\epsilon,\mu^{\epsilon}})-K(z-\bar{X}_{s}^{\epsilon,\mu^{\epsilon}})|\mathcal{F}_{s}^{B}]\psi(x)\pi(dx,dz)\Big\|\nonumber\\
&&+\Big\|\int_{\mathbb{R}^{d}\times\mathbb{R}^{d}}\gamma(z)^{-1}\mathbb{E}[K(z-\bar{X}_{s}^{\epsilon,\mu^{\epsilon}})|\mathcal{F}_{s}^{B}](\psi(x)-\psi(z))\pi(dx,dz)\Big\|\nonumber\\
&\leq&\frac{L_{\gamma}\|\psi\|_{Lip}}{\gamma_{0}^{2}}\int_{\mathbb{R}^{d}\times\mathbb{R}^{d}}\|x-z\|(L_{K}\mathbb{E}[\|x-X_{s}^{\epsilon,\mu^{\epsilon}}\||\mathcal{F}_{s}^{B}]+\|K(0)\|)\pi(dx,dz)\nonumber\\
&&+\frac{L_{K}}{\gamma_{0}}\|\psi\|_{Lip}\int_{\mathbb{R}^{d}\times\mathbb{R}^{d}}\|x-z\|\pi(dx,dz)\nonumber\\
&&+\frac{\|\psi\|_{Lip}}{\gamma_{0}}\int_{\mathbb{R}^{d}\times\mathbb{R}^{d}}\|x-z\|(L_{K}\mathbb{E}[\|z-X_{s}^{\epsilon,\mu^{\epsilon}}\||\mathcal{F}_{s}^{B}]+\|K(0)\|)\pi(dx,dz)\nonumber
\end{eqnarray*}
\begin{eqnarray*}
&\leq&\frac{L_{K}L_{\gamma}\|\psi\|_{Lip}}{\gamma_{0}^{2}}\Big[\Big(\int_{\mathbb{R}^{d}\times\mathbb{R}^{d}}\|x-z\|^{2}\pi(dx,dz)\Big)^{\frac{1}{2}}\Big(\int_{\mathbb{R}^{d}\times\mathbb{R}^{d}}\|x\|^{2}\pi(dx,dz)\Big)^{\frac{1}{2}}\nonumber\\
&&+\Big(\int_{\mathbb{R}^{d}\times\mathbb{R}^{d}}\|x-z\|^{2}\pi(dx,dz)\Big)^{\frac{1}{2}}\mathbb{E}[\|X_{s}^{\epsilon,\mu^{\epsilon}}\||\mathcal{F}_{s}^{B}]\Big]\nonumber\\
&&+\frac{L_{\gamma}L_{K}\|\psi\|_{\infty}\|K(0)\|}{\gamma_{0}^{2}}\Big(\int_{\mathbb{R}^{d}\times\mathbb{R}^{d}}\|x-z\|^{2}\pi(dx,dz)\Big)^{\frac{1}{2}}\nonumber\\
&&+\frac{L_{K}\|\psi\|_{Lip}}{\gamma_{0}}\Big(\int_{\mathbb{R}^{d}\times\mathbb{R}^{d}}\|x-z\|^{2}\pi(dx,dz)\Big)^{\frac{1}{2}}\nonumber\\
&&+\frac{L_{K}\|\psi\|_{Lip}}{\gamma_{0}}\Big[\Big(\int_{\mathbb{R}^{d}\times\mathbb{R}^{d}}\|x-z\|^{2}\pi(dx,dz)\Big)^{\frac{1}{2}}\Big(\int_{\mathbb{R}^{d}\times\mathbb{R}^{d}}\|z\|^{2}\pi(dx,dz)\Big)^{\frac{1}{2}}\nonumber\\
&&+\Big(\int_{\mathbb{R}^{d}\times\mathbb{R}^{d}}\|x-z\|^{2}\pi(dx,dz)\Big)^{\frac{1}{2}}\mathbb{E}[\|X_{s}^{\epsilon,\mu^{\epsilon}}\||\mathcal{F}_{s}^{B}]\nonumber\\
&&+\frac{L_{\gamma}\|K(0)\|\|\psi\|_{Lip}}{\gamma_{0}}\Big(\int_{\mathbb{R}^{d}\times\mathbb{R}^{d}}\|x-z\|^{2}\pi(dx,dz)\Big)^{\frac{1}{2}},\nonumber
\end{eqnarray*}
then by the H${\rm\ddot{o}}$lder inequality, 
\begin{eqnarray}\label{L1-2}
\mathbb{E}L_{1,2}^{\epsilon}(s)&\leq& C_{L_{K},L_{\gamma},\gamma_{0}}\|\psi\|_{Lip}\mathbb{E}\|X_{s}^{\epsilon,\mu^{\epsilon}}-X_{t_{k}}^{\epsilon,\mu^{\epsilon}}\|^{2}[1+\mathbb{E}\|X_{s}^{\epsilon,\mu^{\epsilon}}\|^{2}\nonumber\\
&&+\mathbb{E}\|X_{t_{k}}^{\epsilon,\mu^{\epsilon}}\|^{2}].
\end{eqnarray}
Combining (\ref{L1}) with (\ref{L1-1}), (\ref{L1-2}), Lemma \ref{SUB} and Lemma \ref{SC}, 
\begin{eqnarray}\label{L1Z}
\mathbb{E}L_{1}^{\epsilon}(s)\leq\delta C_{T,X_{0},V_{0},L_{K},L_{\gamma},\gamma_{0}}\|\psi\|_{Lip}.
\end{eqnarray}
Similarly, by $(\mathbf{H_{3}})$, we have
\begin{eqnarray}\label{L2}
\mathbb{E}L_{2}^{\epsilon}(s)&=&\mathbb{E}\Big|\Big\langle \frac{1}{2}\sum_{k=1}^{\infty}\|\sigma_{k}(x)\|^{2}\nabla_{x}\gamma(x)\gamma(x)^{-3}(\rho_{t_{k}}^{\epsilon}-\rho_{s}^{\epsilon}),\psi\Big\rangle\Big|\nonumber\\
&\leq&\frac{\frac{1}{2}L_{\gamma}\sup_{x}\sum_{l=1}^{\infty}\|\sigma_{l}(x)\|^{2}}{\gamma_{0}^{3}}\mathbb{E}|\langle\rho_{t_{k}}^{\epsilon}-\rho_{s}^{\epsilon},\psi\rangle|\nonumber\\
&\leq&C_{\gamma_{0},L_{\sigma},L_{\gamma}}\mathbb{E}\|\mathbb{E}[\|X_{t_{k}}^{\epsilon,\mu^{\epsilon}}-X_{s}^{\epsilon,\mu^{\epsilon}}\||\mathcal{F}_{s}^{B}]\|\|\psi\|_{Lip}\nonumber\\
&\leq&C_{T, X_{0},V_{0},\gamma_{0},L_{\sigma},L_{\gamma}}\|\psi\|_{Lip}\sqrt{\delta},
\end{eqnarray}
and
\begin{eqnarray}\label{L3}
\mathbb{E}L_{3}^{\epsilon}(s)&=&\mathbb{E}\Big|\Big\langle \frac{1}{2}\sum_{k=1}^{\infty}\|\sigma_{k}(x)\|^{2}(\rho_{t_{k}}^{\epsilon}-\rho_{s}^{\epsilon})\gamma(x)^{-2},\psi\Big\rangle\Big|\nonumber\\
&\leq&C_{\gamma_{0},L_{\sigma}}\mathbb{E}\|X_{t_{k}}^{\epsilon,\mu^{\epsilon}}-X_{s}^{\epsilon,\mu^{\epsilon}}\|\|\psi\|_{Lip}\nonumber\\
&\leq&C_{\gamma_{0},L_{\sigma}}\|\psi\|_{Lip}\sqrt{\delta}.
\end{eqnarray}
By (\ref{F4-2}), $(\mathbf{H_{2}})$ and $(\mathbf{H_{3}})$, as $\delta\to0$,
\begin{eqnarray}\label{SJ}
&&\mathbb{E}\Big|\sum_{k=0}^{\lfloor \frac{t}{\delta}\rfloor}\int_{t_{k}}^{t_{k+1}}\Big\langle\sum_{l=1}^{\infty}\sigma_{l}(x)\gamma(x)^{-1}(\rho_{t_{k}}^{\epsilon}-\rho_{s}^{\epsilon}),\psi\Big\rangle dB_{s}^{l}\Big|\nonumber\\
&=&\mathbb{E}\Big|\sum_{l=1}^{\infty}\Big\langle\sigma_{l}(x)\gamma(x)^{-1}\sum_{k=0}^{\lfloor \frac{t}{\delta}\rfloor}\rho_{t_{k}}^{\epsilon}(B_{t_{k+1}}^{l}-B_{t_{k}}^{l})\Big\rangle-\sum_{l=1}^{\infty}\Big\langle\sigma_{l}(x)\gamma(x)^{-1}\int_{0}^{t}\rho_{s}^{\epsilon}dB_{s}^{l},\psi\Big\rangle\Big|\nonumber\\
&\leq&\frac{\|\psi\|_{Lip}}{\gamma_{0}}\sum_{l=1}^{\infty}\|\sigma_{l}(x)\|\Big(\mathbb{E}\Big|\sum_{k=0}^{\lfloor \frac{t}{\delta}\rfloor}\rho_{t_{k}}^{\epsilon}(B_{t_{k+1}}^{l}-B_{t_{k}}^{l})-\int_{0}^{t}\rho_{s}^{\epsilon}dB_{s}^{l}\Big|^{2}\Big)^{\frac{1}{2}}\nonumber\\
&\leq&C_{\gamma_{0},L_{\sigma}}\|\psi\|_{Lip}o(\delta)^{\frac{1}{2}}.
\end{eqnarray}
Combining (\ref{FDH1}) with (\ref{L1Z}), (\ref{L2}), (\ref{L3}) and (\ref{SJ}), we obtain
\begin{eqnarray}\label{FDH2}
&&\mathbb{E}\Big|\sum_{k=0}^{\lfloor \frac{t}{\delta}\rfloor-1}\int_{t_{k}}^{t_{k+1}}\langle Y^{*,\rho_{t_{k}}^{\epsilon}}-Y^{*,\rho_{s}^{\epsilon}}, \psi\rangle ds+\int_{t_{\lfloor \frac{t}{\delta}\rfloor}}^{t}\langle Y^{*,\rho^{\epsilon}_{t_{\lfloor \frac{t}{\delta}\rfloor}}}-Y^{*,\rho_{s}^{\epsilon}},\psi\rangle ds\Big|\nonumber\\
&&\leq C_{\gamma_{0},L_{K},T}\sqrt{\delta}\|\psi\|_{Lip}+C_{\gamma_{0},\|\psi\|_{\infty},L_{\sigma}}o(\delta)^{\frac{1}{2}}.
\end{eqnarray}
Choosing $\delta=\mathcal{O}(\epsilon^{3})$ in (\ref{LL1}), (\ref{F4-6}), (\ref{FDH}) and (\ref{FDH2}), we obtain
\begin{eqnarray*}
\mathbb{E}\langle\partial_{t}\rho_{t},\phi\rangle=-\mathbb{E}\langle\nabla_{x}\cdot Y^{*,\rho_{t}}_{t},\phi\rangle.
\end{eqnarray*}

\end{document}